\documentclass[12pt,reqno]{smfart}
\usepackage{amsmath,amssymb,smfthm,stmaryrd,epic,enumerate}
\usepackage[frenchb]{babel}
\usepackage[latin1]{inputenc}
\usepackage{a4wide}

\usepackage{pstricks}

\usepackage{xypic}
\usepackage[active]{srcltx}
\usepackage{graphicx}

\usepackage{lmodern}

\usepackage{hyperref}

\NumberTheoremsIn{subsection}\SwapTheoremNumbers

\begin{document}
\newtheorem{prop-defi}[smfthm]{Proposition-DÈfinition}
\newtheorem{notas}[smfthm]{Notations}
\newtheorem{nota}[smfthm]{Notation}
\newtheorem{defis}[smfthm]{DÈfinitions}
\newtheorem{hypo}[smfthm]{HypothËse}
\newtheorem*{theo*}{ThÈorËme}
\newtheorem*{hyp*}{HypothËses}

\def\Tm{{\mathbb T}}
\def\Um{{\mathbb U}}
\def\Am{{\mathbb A}}
\def\Fm{{\mathbb F}}
\def\Mm{{\mathbb M}}
\def\Nm{{\mathbb N}}
\def\Pm{{\mathbb P}}
\def\Qm{{\mathbb Q}}
\def\Zm{{\mathbb Z}}
\def\Dm{{\mathbb D}}
\def\Cm{{\mathbb C}}
\def\Rm{{\mathbb R}}
\def\Gm{{\mathbb G}}
\def\Lm{{\mathbb L}}
\def\Km{{\mathbb K}}
\def\Om{{\mathbb O}}
\def\Em{{\mathbb E}}

\def\BC{{\mathcal B}}
\def\QC{{\mathcal Q}}
\def\TC{{\mathcal T}}
\def\ZC{{\mathcal Z}}
\def\AC{{\mathcal A}}
\def\CC{{\mathcal C}}
\def\DC{{\mathcal D}}
\def\EC{{\mathcal E}}
\def\FC{{\mathcal F}}
\def\GC{{\mathcal G}}
\def\HC{{\mathcal H}}
\def\IC{{\mathcal I}}
\def\JC{{\mathcal J}}
\def\KC{{\mathcal K}}
\def\LC{{\mathcal L}}
\def\MC{{\mathcal M}}
\def\NC{{\mathcal N}}
\def\OC{{\mathcal O}}
\def\PC{{\mathcal P}}
\def\UC{{\mathcal U}}
\def\VC{{\mathcal V}}
\def\XC{{\mathcal X}}
\def\SC{{\mathcal S}}
\def\RC{{\mathcal R}}

\def\BF{{\mathfrak B}}
\def\AF{{\mathfrak A}}
\def\GF{{\mathfrak G}}
\def\EF{{\mathfrak E}}
\def\CF{{\mathfrak C}}
\def\DF{{\mathfrak D}}
\def\JF{{\mathfrak J}}
\def\LF{{\mathfrak L}}
\def\MF{{\mathfrak M}}
\def\NF{{\mathfrak N}}
\def\XF{{\mathfrak X}}
\def\UF{{\mathfrak U}}
\def\KF{{\mathfrak K}}
\def\FF{{\mathfrak F}}

\def \longmapright#1{\smash{\mathop{\longrightarrow}\limits^{#1}}}
\def \mapright#1{\smash{\mathop{\rightarrow}\limits^{#1}}}
\def \lexp#1#2{\kern \scriptspace \vphantom{#2}^{#1}\kern-\scriptspace#2}
\def \linf#1#2{\kern \scriptspace \vphantom{#2}_{#1}\kern-\scriptspace#2}
\def \linexp#1#2#3 {\kern \scriptspace{#3}_{#1}^{#2} \kern-\scriptspace #3}

\def \Ext{\mathop{\mathrm{Ext}}\nolimits}
\def \ad{\mathop{\mathrm{ad}}\nolimits}
\def \sh{\mathop{\mathrm{Sh}}\nolimits}
\def \irr{\mathop{\mathrm{Irr}}\nolimits}
\def \FH{\mathop{\mathrm{FH}}\nolimits}
\def \FPH{\mathop{\mathrm{FPH}}\nolimits}
\def \cofil{\mathop{\mathrm{coFil}}\nolimits}
\def \res{\mathop{\mathrm{res}}\nolimits}
\def \op{\mathop{\mathrm{op}}\nolimits}
\def \rec {\mathop{\mathrm{rec}}\nolimits}
\def \art{\mathop{\mathrm{Art}}\nolimits}
\def \hyp {\mathop{\mathrm{Hyp}}\nolimits}
\def \cusp {\mathop{\mathrm{Cusp}}\nolimits}
\def \scusp {\mathop{\mathrm{Scusp}}\nolimits}
\def \Iw {\mathop{\mathrm{Iw}}\nolimits}
\def \JL {\mathop{\mathrm{JL}}\nolimits}
\def \speh {\mathop{\mathrm{Speh}}\nolimits}
\def \isom {\mathop{\mathrm{Isom}}\nolimits}
\def \Vect {\mathop{\mathrm{Vect}}\nolimits}
\def \groth {\mathop{\mathrm{Groth}}\nolimits}
\def \hom {\mathop{\mathrm{Hom}}\nolimits}
\def \deg {\mathop{\mathrm{deg}}\nolimits}
\def \val {\mathop{\mathrm{val}}\nolimits}
\def \det {\mathop{\mathrm{det}}\nolimits}
\def \rep {\mathop{\mathrm{Rep}}\nolimits}
\def \spec {\mathop{\mathrm{Spec}}\nolimits}
\def \fr {\mathop{\mathrm{Fr}}\nolimits}
\def \frob {\mathop{\mathrm{Frob}}\nolimits}
\def \ker {\mathop{\mathrm{Ker}}\nolimits}
\def \im {\mathop{\mathrm{Im}}\nolimits}
\def \Red {\mathop{\mathrm{Red}}\nolimits}
\def \red {\mathop{\mathrm{red}}\nolimits}
\def \aut {\mathop{\mathrm{Aut}}\nolimits}
\def \diag {\mathop{\mathrm{diag}}\nolimits}
\def \spf {\mathop{\mathrm{Spf}}\nolimits}
\def \Def {\mathop{\mathrm{Def}}\nolimits}
\def \twist {\mathop{\mathrm{Twist}}\nolimits}
\def \supp {\mathop{\mathrm{Supp}}\nolimits}
\def \Id {{\mathop{\mathrm{Id}}\nolimits}}
\def \lie {{\mathop{\mathrm{Lie}}\nolimits}}
\def \Ind{\mathop{\mathrm{Ind}}\nolimits}
\def \ind {\mathop{\mathrm{ind}}\nolimits}
\def \soc {\mathop{\mathrm{Soc}}\nolimits}
\def \top {\mathop{\mathrm{Top}}\nolimits}
\def \ker {\mathop{\mathrm{Ker}}\nolimits}
\def \coker {\mathop{\mathrm{Coker}}\nolimits}
\def \gal {{\mathop{\mathrm{Gal}}\nolimits}}
\def \Nr {{\mathop{\mathrm{Nr}}\nolimits}}
\def \rn {{\mathop{\mathrm{rn}}\nolimits}}
\def \tr {{\mathop{\mathrm{Tr~}}\nolimits}}
\def \Sp {{\mathop{\mathrm{Sp}}\nolimits}}
\def \st {{\mathop{\mathrm{St}}\nolimits}}
\def \sp{{\mathop{\mathrm{Sp}}\nolimits}}
\def \perv{\mathop{\mathrm{Perv}}\nolimits}
\def \tor {{\mathop{\mathrm{Tor}}\nolimits}}
\def \nrd {{\mathop{\mathrm{Nrd}}\nolimits}}
\def \nilp {{\mathop{\mathrm{Nilp}}\nolimits}}
\def \obj {{\mathop{\mathrm{Obj}}\nolimits}}
\def \spl {{\mathop{\mathrm{Spl}}\nolimits}}
\def \gr {{\mathop{\mathrm{gr}}\nolimits}}

\def \rem{{\noindent\textit{Remarque.~}}}
\def \rems{{\noindent\textit{Remarques:~}}}
\def \ext {{\mathop{\mathrm{Ext}}\nolimits}}
\def \End {{\mathop{\mathrm{End}}\nolimits}}

\def\semi{\mathrel{>\!\!\!\triangleleft}}
\let \DS=\displaystyle

\def\HT{{\mathop{\mathcal{HT}}\nolimits}}

\setcounter{secnumdepth}{3} \setcounter{tocdepth}{3}

\def \Fil{\mathop{\mathrm{Fil}}\nolimits}
\def \CoFil{\mathop{\mathrm{CoFil}}\nolimits}
\def \Fill{\mathop{\mathrm{Fill}}\nolimits}
\def \CoFill{\mathop{\mathrm{CoFill}}\nolimits}
\def\SF{{\mathfrak S}}
\def\PF{{\mathfrak P}}
\def \EFil{\mathop{\mathrm{EFil}}\nolimits}
\def \EFill{\mathop{\mathrm{EFill}}\nolimits}
\def \FP{\mathop{\mathrm{FP}}\nolimits}

\let \longto=\longrightarrow
\let \oo=\infty

\let \d=\delta
\let \k=\kappa

\def \hi{\HC}

\newcommand{\marque}{\addtocounter{smfthm}{1}
{\smallskip \noindent \textit{\thesmfthm}~---~}}

\renewcommand\atop[2]{\ensuremath{\genfrac..{0pt}{1}{#1}{#2}}}

\title{Principe de Mazur en dimension supÈrieure}

\alttitle{Mazur's principle in higher dimension}

\author{Boyer Pascal}
\email{boyer@math.univ-paris13.fr}
\address{UniversitÈ Paris 13, Sorbonne Paris CitÈ \\
LAGA, CNRS, UMR 7539\\ 
F-93430, Villetaneuse (France) \\
PerCoLaTor: ANR-14-CE25}
%

\thanks{L'auteur remercie l'ANR pour son soutien dans le cadre du projet PerCoLaTor 14-CE25.}

\frontmatter

\begin{abstract}
Le principe de Mazur pour $GL_2$ fournit des conditions simples pour qu'une $\overline \Fm_l$-reprÈsentation
irrÈductible non ramifiÈe provenant d'une forme modulaire de niveau $\Gamma_0(Np)$ provienne aussi
d'une forme de niveau $\Gamma_0(N)$. L'objectif de ce travail est de proposer une gÈnÈralisation de
ce principe en dimension supÈrieure pour certaines formes intÈrieures Ètendues non quasi-dÈployÈes
d'un groupe unitaire en Ètudiant la torsion dans la cohomologie 
des variÈtÈs de Shimura dites de Kottwitz-Harris-Taylor en lien avec la dÈgÈnÈrescence de la monodromie locale.

\end{abstract}

\begin{altabstract}
The Mazur principle for $GL_2$ gives simple conditions for an irreducible unramified $\overline \Fm_l$-representation
coming from a modular form of level $\Gamma_0(Np)$ to come for some modular form of level 
$\Gamma_0(N)$. The aim of this work is to give a generalization of this principle in higher dimension for
some particular extended inner forms non quasi split of a unitary group
studying the torsion cohomology classes of Shimura varieties of Kottwitz-Harris-Taylor type within its link with
the local monodromy degeneracy.

\end{altabstract}



\altkeywords{Shimura varieties, torsion in the cohomology, maximal ideal of the Hecke algebra,
localized cohomology, galois representation}

%
%
%
%
%
%
%
%

\maketitle

\pagestyle{headings} \pagenumbering{arabic}

\tableofcontents
%
%

\section*{Introduction}

Dans la thÈorie classique des formes modulaires, une question importante est la dÈtermination
du niveau optimal ‡ partir duquel
une reprÈsentation galoisienne modulo $l$ est modulaire: que l'on pense par exemple ‡ son
application ‡ la preuve du grand thÈorËme de Fermat. Les conjectures de Serre, dÈsormais prouvÈes par
Khare et Wintenberger dans \cite{KW}, fournissent un cadre prÈcis pour cette question dans le cas de 
$GL_2$. Avant que ne soient Ètablies les conjectures de Serre, le principe de Mazur, rappelÈ ci-aprËs, 
constituait le rÈsultat le plus ÈvoluÈ sur ce thËme et, par exemple, l'ingrÈdient principal dans la preuve du 
thÈorËme de Ribet.

\begin{theo*}\textbf{(Principe de Mazur cf. \cite{ribet} thÈorËme 6.1)} \\
Soient $N$ un entier, $p$ un nombre premier ne divisant pas $N$ et $\bar \rho: \gal(\overline \Qm_l/\Qm)
\longrightarrow GL_2(\overline \Fm_l)$ une reprÈsentation galoisienne provenant d'une forme modulaire
de niveau $\Gamma_0(Np)$. On suppose que
\begin{itemize}
\item $p \neq l$,

\item $\bar \rho$ est irrÈductible et non ramifiÈe en $p$ et

\item $l$ ne divise pas $p-1$.
\end{itemize}
Alors $\bar \rho$ provient d'une forme modulaire de niveau $\Gamma_0(N)$.
\end{theo*}

L'objectif de ce travail est de proposer une version du principe de Mazur pour les reprÈsentations
automorphes de $GL_d$
autoduales pour un corps CM, en utilisant la cohomologie des variÈtÈs de
Shimura. Il est bien connu qu'au del‡ du cas $d=2$, il n'y a pas de variÈtÈ de Shimura pour $GL_d$ et la solution usuelle consiste
‡ remplacer $GL_d$ par un groupe de similitudes $G/\Qm$ qui, localement pour \og la moitiÈ \fg{} des premiers $p$, ressemble ‡
$GL_d(\Qm_p)$, au sens plus prÈcis o˘, cf. (\ref{eq-facteur-v}), 
$G(\Qm_p) \simeq \Qm_p^\times \times GL_d(F_v) \times \cdots$, o˘ $F$ est un corps CM
et o˘ $v|p$ sera une place de $F$ jouant le rÙle du premier $p$ dans le principe de Mazur. 
Pour que la situation gÈomÈtrique soit la plus simple possible et qu'on dispose donc d'un meilleur contrÙle de la cohomologie,
on choisit le groupe $G$ de faÁon ‡ nous retrouver dans la situation ÈtudiÈe par Harris et Taylor dans \cite{h-t}, 
i.e. en signatures $(1,d-1)\times (0,d) \times \cdots \times (0,d)$.
La formulation prÈcise 
du principe de Mazur dans cette  situation est donnÈe au thÈorËme \ref{theo-principal}, donnons simplement dans cette introduction
une idÈe de ce que devient l'hypothËse clef \og $\bar \rho$ non ramifiÈe en $p$ \fg{} dans notre situation.

Une reprÈsentation automorphe $\Pi$ 
de $G$ fournit des paramËtres de Satake en ses places de non ramification et donc
un idÈal premier $\widetilde{\mathfrak m}$ d'une algËbre de Hecke \og anÈmique \fg, 
cf. la dÈfinition \ref{nota-spl2}: on note aussi $\mathfrak m$ l'idÈal maximal associÈ ‡ la rÈduction modulo $l$ de ces paramËtres de
Satake. D'aprËs \cite{h-t}, on associe ‡
$\widetilde{\mathfrak m}$ une reprÈsentation $ \rho_{ \widetilde{\mathfrak m} }$ de $\gal(\bar F/F)$ et $\bar \rho_{\mathfrak m}$ sa rÈduction
modulo $l$. Comme dans le cas de $GL_2$
\begin{itemize}
\item on part d'une $\overline{\mathbb F}_l$-reprÈsentation $\bar \rho_{\mathfrak m}$ de $\gal(\bar F/F)$ s'Ècrivant comme la rÈduction
modulo $l$  de $ \rho_{ \widetilde{\mathfrak m} }$ o˘ $ \widetilde{\mathfrak m}$ est associÈ ‡ une reprÈsentation automorphe 
$\Pi$ de niveau $I$, 

\item et on cherche des conditions pour l'existence d'une reprÈsentation automorphe $\Pi'$ de niveau $I'$ avec $I_v \subsetneq I'_v$, quitte ‡ augmenter $I$ en des places annexes $w \neq v$, de sorte que si $ \widetilde{\mathfrak m}'$ est  l'idÈal premier de l'algËbre de Hecke 
anÈmique associÈ ‡ $\Pi'$, alors $\widetilde{\mathfrak m}' \subset \mathfrak m$, autrement dit si $ \rho_{ \widetilde{\mathfrak m}',v}$ est la reprÈsentation galoisienne construite par \cite{h-t}, alors
sa rÈduction modulo $l$ est isomorphe ‡ $\bar \rho_{\mathfrak m}$.
\end{itemize}
Lorsque la composante en $p$ du sous-groupe compact considÈrÈ est parahorique,
la condition de non ramification en $p$ dans le cas de $GL_2$, est remplacÈe par la dÈgÈnÈrescence
de la monodromie au sens suivant. Le logarithme de la monodromie ‡ la place $v$ de $ \rho_{ \widetilde{\mathfrak m} }$
dÈfinit un opÈrateur nilpotent $N_{ \widetilde{\mathfrak m},v}$ de $GL_d$ dont la taille des blocs de Jordan fournit 
une partition $ \underline{d_{ \widetilde{\mathfrak m},v}}$ de $d$. En supposant $\bar \rho_{\mathfrak m}$ irrÈductible et en prenant $l \geq d_{\widetilde{\mathfrak m},v}-1$,
o˘ $d_{\widetilde{\mathfrak m},v}$ est l'indice de nilpotence de $N_{ \widetilde{\mathfrak m},v}$,
alors $N_{ \widetilde{\mathfrak m},v}$ possËde une structure entiËre unique et admet
donc une rÈduction modulo $l$ fournissant une partition $ \underline{d_{\mathfrak m,v}}$ ne dÈpendant pas du choix de $\Pi$. 
La condition de non ramification pour $GL_2$ devient alors: \emph{pour la relation
de dominance usuelle sur les partitions, $ \underline{d_{\mathfrak m,v}}$ est strictement plus 
petite que $\underline{d_{ \widetilde{\mathfrak m},v}}$.}

La dÈmonstration repose sur l'Ètude de la torsion dans la cohomologie des
variÈtÈs de Shimura dites de Kottwitz-Harris-Taylor, et sur l'observation que cette torsion
se relËve, cf. le rÈsultat principal de \cite{boyer-mrl}, en caractÈristique $0$ quitte ‡ augmenter le niveau en une place annexe.
L'idÈe consiste alors ‡ jouer avec cette propriÈtÈ 
\begin{itemize}
\item en la place $p$ o˘  ‡ l'aide des hypothËses du thÈorËme \ref{theo-principal}, on parvient ‡ 
diminuer le niveau en $p$ tout en gardant une torsion non triviale,

\item puis en augmentant le niveau en une place annexe quelconque, on relËve cette torsion en 
caractÈristique nulle.
\end{itemize}

On Ètudie en outre, cf. le corollaire \ref{coro-principal}, l'existence d'un $\Pi$ tel que 
$ \underline{d_{\mathfrak m,v}} = \underline{d_{ \widetilde{\mathfrak m},v}}$
ainsi, cf. le corollaire \ref{coro-appli}, que des conditions explicites sur $ \widetilde{\mathfrak m}$ pour que 
$N_{ \widetilde{\mathfrak m},v}$ en $v$ 
ne dÈgÈnËre pas i.e. tel que la partition en bloc de Jordan de la rÈduction modulo $l$ de $N_{ \widetilde{\mathfrak m},v}$ soit Ègale ‡
$ \underline{d_{ \widetilde{\mathfrak m},v}}$.

Nous remercions V. SÈcherre pour nous avoir expliquÈ le lemme \ref{lem-secherre}. ainsi que V. Lafforgue pour ses nombreuses remarques
sur une premiËre version de ce travail.

\section{DÈgÈnÈrescence de la monodromie et diminution du niveau}

\renewcommand{\thesmfthm}{\arabic{section}.\arabic{subsection}.\arabic{smfthm}}

\renewcommand{\theequation}{\arabic{section}.\arabic{subsection}.\arabic{smfthm}}


\subsection{Rappels sur les $\overline \Qm_l$-reprÈsentations de $GL_d(K)$}

Notons $K$ un corps local non archimÈdien dont le corps rÈsiduel est de cardinal $q$ une puissance 
d'un nombre premier $p$. Une racine carrÈe $q^{\frac{1}{2}}$ de $q$ dans $\overline \Qm_l$ Ètant fixÈe,
pour $k \in \frac{1}{2} \Zm$, nous noterons $\pi\{ k \}$ la reprÈsentation tordue de $\pi$ o˘ l'action 
de $g \in GL_n(K)$ est donnÈe par $\pi(g) \nu(g)^k$ avec
$\nu: g \in GL_n(K) \mapsto q^{-\val (\det g)}$.

\begin{defis}
Soit $P=MN$ un parabolique standard de $GL_n$ de LÈvi $M$ et de radical unipotent $N$.
On note $\delta_P:P(K) \rightarrow \overline \Qm_l^\times$ l'application dÈfinie par
$$\delta_P(h)=|\det (\ad(h)_{|\lie N})|^{-1}.$$
Pour $(\pi_1,V_1)$ et $(\pi_2,V_2)$ des reprÈsentations de respectivement $GL_{n_1}(K)$ 
et $GL_{n_2}(K)$, et $P_{n_1,n_2}$ le parabolique standard de $GL_{n_1+n_2}$ de Levi 
$M=GL_{n_1} \times GL_{n_2}$ et de radical unipotent $N$, 
$$\pi_1 \times \pi_2$$
dÈsigne l'induite parabolique normalisÈe de $P_{n_1,n_2}(K)$ ‡ $GL_{n_1+n_2}(K)$ de 
$\pi_1 \otimes \pi_2$ c'est ‡ dire
l'espace des fonctions $f:GL_{n_1+n_2}(K) \rightarrow V_1 \otimes V_2$ telles que
$$f(nmg)=\delta_{P_{n_1,n_2}}^{-1/2}(m) (\pi_1 \otimes \pi_2)(m) \Bigl ( f(g) \Bigr ),
\quad \forall n \in N, ~\forall m \in M, ~ \forall g \in GL_{n_1+n_2}(K).$$
\end{defis}
%

Rappelons qu'une reprÈsentation irrÈductible $\pi$ de $GL_n(K)$ est dite \textit{cuspidale} 
si elle n'est pas isomorphe ‡ un sous-quotient d'une induite parabolique propre.

\begin{nota}
Soient $g$ un diviseur de $d=sg$ et $\pi$ une reprÈsentation cuspidale
irrÈductible de $GL_g(K)$. L'unique quotient (resp. sous-reprÈsentation) irrÈductible de
$\pi\{ \frac{1-s}{2} \} \times \pi\{\frac{3-s}{2} \} \times \cdots \times \pi\{ \frac{s-1}{2} \}$
est notÈ $\st_s(\pi)$ (resp. $\speh_s(\pi)$).
\end{nota}

Notons $\OC_K$ l'anneau des entiers de $K$ et $\varpi_K$ une uniformisante. 
Le sous-groupe compact ouvert de $GL_d(\OC_K)$ 
des ÈlÈments dont la rÈduction modulo 
$\varpi_K$ est triangulaire supÈrieure, est le classique sous-groupe d'Iwahori.
Rappelons que toute reprÈsentation irrÈductible de $GL_d(K)$ admettant des vecteurs non nuls
invariants sous le sous-groupe d'Iwahori est, avec les notations prÈcÈdentes 
un sous-quotient d'une induite $\chi_{1} \times \cdots \times \chi_{d}$ o˘ les
$\chi_{i}$ sont des caractËres de $K^\times$ uniquement dÈfinis ‡ l'ordre prËs.

\begin{nota} 
Pour toute reprÈsentation irrÈductible $\pi$ de $GL_d(K)$ ayant des vecteurs non nuls
invariants par le sous-groupe d'Iwahori, on note\footnote{Plus prÈcisÈment $V(\pi)$ est un 
multi-ensemble, i.e. on garde en mÈmoire la rÈpÈtition des $\chi(\varpi_K)$.}
$$V(\pi)= \Bigl \{ \chi_{i}(\varpi_K):~i=1,\cdots,d \Bigr \}$$
o˘ les caractËres $\chi_{i}$ sont tels que $\pi$ est un sous-quotient de l'induite
$\chi_{1} \times \cdots \times \chi_{d}$.
\end{nota}

\rem Avec les notations prÈcÈdentes, on a
$$V(\st_{t}(\chi))= \Bigl \{ \chi(\varpi_K)q^{\frac{1-t}{2}}, \chi(\varpi_K)q^{\frac{1-t+2}{2}},\cdots,
\chi(\varpi_K)q^{\frac{t-1}{2}} \Bigr \}.$$

\begin{defi}
…tant donnÈe une partition de $d$
$$\underline m=(m_1 \geq m_2 \geq \cdots \geq m_r \geq 1) \hbox{ avec }
d=m_1+\cdots+m_r,$$ 
le sous-groupe parahorique standard associÈ $\Iw(\underline m)$ est par dÈfinition l'ensemble des ÈlÈments de 
$GL_d(\OC_K)$ dont la rÈduction modulo $\varpi_K$ appartient au parabolique standard $P_{m_1,\cdots,m_r}$ de Levi
$GL_{m_1} \times GL_{m_2} \times \cdots \times GL_{m_r}$:
$$\Iw(\underline m):=\ker \Bigl ( GL_d(\OC_K) \longrightarrow P_{m_1,m_2,\cdots,m_r}(\OC_K/(\varpi_K)) \Bigr ).$$
Un sous-groupe parahorique est alors un conjuguÈ d'un parahorique
standard par un ÈlÈment de $GL_d(\OC_K)$.
\end{defi}

\rem Pour $m_1=m_2=\cdots=m_d=1$, on retrouve le classique sous-groupe d'Iwahori.

On associe habituellement ‡ une partition $\underline m=(m_1 \geq \cdots \geq m_r)$ de $d$, 
un diagramme de Ferrers
dont la $i$-Ëme ligne est de longueur $m_i$. Les longueurs $t_1 \geq \cdots \geq t_{m_1}$
des colonnes du diagramme de Ferrers dÈfinissent alors la partition 
$$\underline m^*=(t_1 \geq \cdots \geq t_{m_1})$$
conjuguÈe de $\underline m$. 

\begin{nota} \label{nota-partition1}
Pour $\underline m$ une partition de conjuguÈe $\underline m^*=(t_1 \geq t_2 \geq \cdots \geq t_{m_1})$, 
on note $\underline m^{(1)}$ la partition dont la conjuguÈe est $(t_2 \geq t_3 \geq \cdots \geq t_{m_1})$.
\end{nota}

\rem Autrement dit $\underline m^{(1)}$ est la partition obtenue ‡ partir de $\underline m$ en supprimant
sa premiËre colonne.

\begin{defi} \label{defi-partition-ordre}
 On dira d'une partition $\underline m=(m_1 \geq m_2 \geq \cdots \geq m_r)$ de $n$ 
qu'elle est contenue dans une partition $\underline m'=(m'_1 \geq m'_2 \geq \cdots \geq m'_{r'})$ de $n'$
si $r \leq r'$ et si pour tout $i=1,\cdots,r$ on a $m_i \leq m'_i$.
\end{defi}

On rappelle la relation de dominance usuelle sur les partitions 
$$\underline n=(n_1 \geq n_2 \geq \cdots) \leq \underline m=(m_1 \geq m_2 \geq \cdots) \Leftrightarrow
\forall k \geq 1 \sum_{i=1}^k n_i \leq \sum_{i=1}^k m_i.$$
La relation $\underline n \leq \underline m$ est Èquivalente ‡ 
$\underline m^* \leq \underline n^*$ sur les partitions conjuguÈes.

\begin{lemm} \label{lem-secherre}
Soit $\pi \simeq \st_{t_1}(\chi_1) \times \cdots \times \st_{t_s}(\chi_s)$ avec $t_1+\cdots+t_s=d$
et o˘ $\chi_1,\cdots,\chi_s$ sont des caractËres de $K^\times$.  
L'ensemble des sous-groupes parahoriques $P$ 
tels que $\pi$ ait des vecteurs non nuls $P$-fixes,
admet un plus grand ÈlÈment  dont les tailles des blocs sont, ‡ conjugaison prËs, ceux de la partition
$\underline d(\pi)$ conjuguÈe ‡ $(t_1 \geq \cdots \geq t_s)$.
\end{lemm}

\begin{proof}
Notons $\KC=GL_d(\OC_K)$ le compact maximal de $GL_d(K)$, puis $\KC(1)$ son pro-$p$ radical.
Rappelons que $\pi$ a des vecteurs non nuls $P$-fixes si et seulement si 
l'espace $\KC(\pi)$ des vecteurs non nuls $\KC(1)$-fixes de $\pi$ vu comme reprÈsentation de 
$\KC/\KC(1)$ a des vecteurs non nuls fixes par $P':=P/\KC(1)$, c'est-‡-dire si $\KC(\pi)^{U'}$, o˘
$U'$ le radical unipotent de $P'$, contient le caractËre trivial de $M'=P'/U'$. 

En appliquant l'involution de Zelevinski $Z$, on se ramËne ‡ la propriÈtÈ que $\KC(Z(\pi))^{U'}$
contient un facteur non dÈgÈnÈrÈ, c'est-‡-dire au fait que $Z(\pi)$ est $\lambda$-dÈgÈnÈrÈe,
o˘ $\lambda$ dÈsigne la partition de $d$ donnÈe par les blocs de $M'$, 
au sens de la thÈorie des modËles de Whittaker dÈgÈnÈrÈs, cf. \cite{vigneras-induced} \S V.5.
Le rÈsultat dÈcoule alors de loc. cit.
\end{proof}

\begin{defi} \label{defi-TY1}
¿ la reprÈsentation $\pi\simeq \st_{t_1}(\chi_1) \times \cdots \times \st_{t_s}(\chi_s)$ on associe
$T(\pi)$ le diagramme de Ferrers ÈtiquetÈ par $V(\pi)$ dÈfini comme suit:
\begin{itemize}
\item les longueurs des lignes de ce tableau sont les $t_1 \geq t_2 \geq \cdots \geq t_s$

\item et on Ètiquette la $i$-Ëme ligne de gauche ‡ droite avec dans l'ordre les
$\chi_{i}q^{\frac{1-t_i}{2}},\cdots,\chi_{i}q^{\frac{t_i-1}{2}}$.
\end{itemize}
\end{defi}

\subsection{ReprÈsentation galoisienne associÈe ‡ $\mathfrak m$}
\label{para-nota}

Soient $F^+$ un corps totalement rÈel et $E/\Qm$ une extension quadratique imaginaire: 
on considËre alors le corps $F=EF^+$ qui est CM. Pour toute place $w$ de $F$, on notera 
\begin{itemize}
\item $F_w$ son localisÈ en $w$, 
\item $\OC_w$ son anneau des entiers d'uniformisante $\varpi_w$ et
\item $q_w$ le cardinal du corps rÈsiduel $\kappa(w):=\OC_w/(\varpi_w)$.
\end{itemize}

Soit $B$ une algËbre ‡ 
division centrale sur $F$ de dimension $d^2$ telle qu'en toute place $x$ de $F$,
$B_x$ est soit dÈcomposÈe soit une algËbre ‡ division et on suppose $B$ 
munie d'une involution de
seconde espËce $*$ telle que $*_{|F}$ est la conjugaison complexe $c$. Pour
$\beta \in B^{*=-1}$, on note $\sharp_\beta$ l'involution $x \mapsto x^{\sharp_\beta}=\beta x^*
\beta^{-1}$ et $G/\Qm$ le groupe de similitudes, notÈ $G_\tau$ dans \cite{h-t}, dÈfini
pour toute $\Qm$-algËbre $R$ par 
$$
G(R)  \simeq   \{ (\lambda,g) \in R^\times \times (B^{op} \otimes_\Qm R)^\times  \hbox{ tel que } 
gg^{\sharp_\beta}=\lambda \}
$$
avec $B^{op}=B \otimes_{F,c} F$. 
Si $x$ est une place de $\Qm$ dÈcomposÈe $x=yy^c$ dans $E$ alors 
\addtocounter{smfthm}{1}
\begin{equation} \label{eq-facteur-v}
G(\Qm_x) \simeq (B_y^{op})^\times \times \Qm_x^\times \simeq \Qm_x^\times \times
\prod_{z_i} (B_{z_i}^{op})^\times,
\end{equation}
o˘, en identifiant les places de $F^+$ au dessus de $x$ avec les places de $F$ au dessus de $y$,
$x=\prod_i z_i$ dans $F^+$.
Dans \cite{h-t} lemme I.7.1, les auteurs justifient l'existence d'un $G$ comme ci-dessus tel que pour tous
tels $d, E^+$ et $F$:
\begin{itemize}
\item si $x$ est une place de $\Qm$ qui n'est pas dÈcomposÈe dans $E$ alors
$G(\Qm_x)$ est quasi-dÈployÈ;

\item les invariants de $G(\Rm)$ sont $(1,d-1)$ pour le plongement $\tau$ et $(0,d)$ pour les
autres. 
\end{itemize}
\emph{On fixe} ‡ prÈsent un nombre premier $p=uu^c$ dÈcomposÈ dans $E$ tel qu'il existe une 
place $v$ de $F$ au dessus de $u$ avec 
$$(B_v^{op})^\times \simeq GL_d(F_v).$$
On note 
$$v_1=v, v_2,\cdots, v_r$$ 
les places de $F$ au dessus de $u$. Avec un abus coupable
de notation, on utilisera $G(F_v)$ pour dÈsigner le
facteur en $v$ de la formule (\ref{eq-facteur-v}), isomorphe donc ‡ $GL_d(F_v)$.

\begin{defi}  \label{defi-Kv}
Soit $\IC$ l'ensemble des sous-groupes compacts ouverts \og assez petits \fg{}\footnote{tels qu'il existe 
une place $x \neq p$ pour laquelle la projection de $U^v$ sur $G(\Qm_x)$ 
ne contienne aucun ÈlÈment d'ordre fini 
autre que l'identitÈ, cf. \cite{h-t} bas de la page 90}  de  $G(\Am^\oo)$, de la forme $U^v K_v$ avec
\begin{itemize}
\item $U^v=U^p \times \Zm_p^\times \times \prod_{i=2}^r 
\ker \bigl ( \OC_{B_{v_i}}^\times \longto (\OC_{B_{v_i}}/\PC_{v_i}^{n_i})^\times \bigr )$
pour des entiers $n_2,\cdots,n_r$ positifs ou nuls,

\item et o˘ $K_v$ est un sous-groupe parahorique. 

\item Pour $I=U^vK_v \in \IC$ comme ci-avant, on notera $I^v=U^v$ et $I_v=K_v$.
\end{itemize}
\end{defi}

On note alors, cf. le \S \ref{para-KHT}, $(X_I)_{I \in \IC}$ 
le systËme projectif des variÈtÈs de Shimura, dites de Kottwitz-Harris-Taylor, 
associÈ au groupe $G$ au dessus de $\spec \OC_v$ tel qu'il
est introduit dans \cite{h-t}.

\begin{nota}
Pour $I \in \IC$, on note $\spl(I)$ l'ensemble des places $w$ de $F$
telles que $p_w:=w_{|\Qm}$ est dÈcomposÈe dans $E$ et, cf. la formule (\ref{eq-facteur-v}), 
la composante locale 
$I_w$ de $I$ ‡ la place $w$ est isomorphe ‡ $GL_d(\OC_w)$.
\end{nota}

Fixons pour la suite un isomorphisme $\iota:\overline \Qm_l \simeq \Cm$.
\'Etant donnÈe une $\overline \Qm_l$-reprÈsentation algÈbrique irrÈductible $\xi$ 
de $G(\Qm)$,
rappelons qu'une $\Cm$-reprÈsentation irrÈductible $\Pi_{\oo}$ de $G(\Am_{\oo})$ est dite 
$\xi$-cohomologique s'il existe un entier $i$ tel que
$$H^i((\lie ~G(\Rm)) \otimes_\Rm \Cm,U_\tau,\Pi_\oo \otimes \iota(\xi^\vee)) \neq (0)$$
o˘ $U_\tau$ est un sous-groupe compact modulo le centre de $G(\Rm)$, maximal, cf. \cite{h-t} p.92. 
Une $\overline \Qm_{l}$-reprÈsentation irrÈductible $\Pi^{\oo}$ de $G(\Am^{\oo})$ sera dit automorphe 
$\xi$-cohomologique s'il existe une $\Cm$-reprÈsentation $\xi$-cohomologique $\Pi_\oo$ de 
$G(\Am_\oo)$ telle que
$\iota\Bigl ( \Pi^{\oo} \Bigr ) \otimes \Pi_{\oo}$ est une $\Cm$-reprÈsentation automorphe de $G(\Am)$.

\rem Si $\iota':\overline \Qm_l \simeq \Cm$ est un autre choix d'isomorphisme et si
$\Pi^\oo$ est automorphe $\xi$-cohomologique relativement ‡ $\iota$, alors, 
d'aprËs la formule de Matsushima,
$(\iota')^{-1} \circ \iota (\Pi^\oo)$ l'est relativement ‡ $\iota'$.

\begin{defi} \label{nota-spl2}
Pour $l$ un nombre premier distinct de $p$ et $I \in \IC$ un niveau fini, soit
$$\Tm_I:=\Zm_l \bigl [T_{w,i}:~w \in \spl(I) \hbox{ et } i=1,\cdots,d \bigr ],$$
l'algËbre de Hecke \og anÈmique \fg{} associÈe ‡ $\spl(I)$, o˘ $T_{w,i}$ est la fonction caractÈristique de
$$GL_d(\OC_w) \diag(\overbrace{\varpi_w,\cdots,\varpi_w}^{i}, \overbrace{1,\cdots,1}^{d-i} ) 
GL_d(\OC_w) \subset  GL_d(F_w).$$
\end{defi}

\rem L'algËbre $\Tm_I$ ne dÈpend que de $\spl (I)$ au sens o˘ si $J \in _IC$ est tel que
$\spl (J)=\spl (I)$ alors $\Tm_J=\Tm_I$.

On fixe ‡ prÈsent une reprÈsentation algÈbrique $\xi$ de $G(\Qm)$ et on note, cf. le \S \ref{para-xi},
$V_{\xi,\overline \Zm_l}$ le $\overline \Qm_l$-systËme local associÈ, dÈfini sur tout 
$X_I$ pour $I \in \IC$.
On note alors $\Tm_I(\xi)$ l'image de $\Tm_I$ dans les endomorphismes $\overline \Zm_l$-linÈaires
du quotient libre $H^{d-1}_{free}(X_{I,\bar \eta_v},V_{\xi,\overline \Zm_l})$ de la cohomologie en
degrÈ mÈdian de la fibre  gÈnÈrique gÈomÈtrique de $X_I$, ‡ coefficients dans 
$V_{\xi,\overline \Zm_l}$.

\rem Dans la suite nous ne nous intÈresserons qu'aux systËmes de valeurs propres de Hecke
$\mathfrak m$ de $\Tm_I$ donnant lieu, cf. le dÈbut du \S \ref{para-enonce}, 
‡ une reprÈsentation galoisienne $\bar \rho_{\mathfrak m}$
irrÈductible de sorte qu'il suffit de considÈrer l'image de $\Tm_I$ dans la cohomologie en degrÈ
mÈdian. En outre on a vu dans \cite{boyer-imj} que tout systËme de valeurs propres de Hecke
dans la $\overline \Fm_l$-cohomologie, se relevait, quitte ‡ augmenter le niveau $I$, en une
reprÈsentation automorphe tempÈrÈe entiËre, ce qui justifie de ne regarder que l'image de $\Tm_I$
dans le quotient libre de la cohomologie en degrÈ mÈdian.

Les idÈaux premiers minimaux de $\Tm_I(\xi)$ sont les idÈaux premiers de $\Tm_I(\xi)$ 
au dessus de l'idÈal nul de $\Zm_l$ et sont donc en
bijection naturelle avec les idÈaux premiers de $\Tm_I(\xi) \otimes_{\Zm_l} \Qm_l$. Ainsi pour
un tel idÈal $\widetilde{\mathfrak m}$ premier minimal, $(\Tm_I(\xi) \otimes_{\Zm_l} \Qm_l)
/\widetilde{\mathfrak m}$ est une extension finie $K_{\widetilde{\mathfrak m}}$ de $\Qm_l$.

\rem
Un idÈal $\widetilde{\mathfrak m}$ premier minimal de $\Tm_I(\xi)$ 
est dit $\xi$-cohomologique au sens o˘ il existe une $\overline \Qm_l$-reprÈsentation automorphe 
$\xi$-cohomologique 
$\Pi$ de $G(\Am)$ possÈdant des vecteurs non nuls fixes sous $I$ et 
telle que pour tout $w \in \spl(I)$, les paramËtres de Satake de $\Pi_{p_w}$ 
sont donnÈs par les images des $T_w^{(i)} \in K_{\widetilde{\mathfrak m}}:=(\Tm_I(\xi) 
\otimes_{\Zm_l} \Qm_l)/\widetilde{\mathfrak m}$,
o˘ $K_{\widetilde{\mathfrak m}}$ est une extension finie de $\Qm_l$.  On notera
q'un tel $\Pi$ n'est pas nÈcessairement unique mais dÈfinit une unique classe d'Èquivalence 
proche au sens de \cite{y-t} que l'on notera $\Pi_{\widetilde{\mathfrak m}}$.

On fixe une clÙture algÈbrique $\overline \Qm_l$ et on notera $\overline \Tm_I(\xi):=\Tm_I(\xi)
\otimes_{\Zm_l}
\overline \Zm_l$ de sorte que $K_{\widetilde{\mathfrak m}}$ s'injecte canoniquement dans 
$\overline \Qm_l=(\overline \Tm_I(\xi) \otimes_{\overline \Zm_l} \overline \Qm_l)/\widetilde{\mathfrak m}
\otimes_{\Zm_l} \overline \Zm_l.$

Dans la suite nous ne considÈrerons que des idÈaux premiers
$\xi$-cohomologiques, ce qui permet de dÈfinir leur reprÈsentation galoisienne associÈe au sens suivant.

\begin{defi}
On note 
$$\rho_{\widetilde{\mathfrak m}}:\gal(\bar F/F) \longrightarrow GL_d(\overline \Qm_l)$$ 
la reprÈsentation galoisienne associÈe ‡ un tel $\Pi$ d'aprËs
\cite{h-t} et \cite{y-t}.
\end{defi}

\rem Pour toute place $w \in \spl(I)$, les valeurs propres de $\rho_{\widetilde{\mathfrak m}}(\frob_w)$
sont donnÈes par les paramËtres de Satake de $\Pi_v$ et appartiennent donc ‡ 
$K_{\widetilde{\mathfrak m}}$. 

La restriction de $\rho_{\widetilde{\mathfrak m}}$
au groupe de Galois local en $v$ s'identifie ‡ une reprÈsentation de Weil-Deligne
$(\sigma_{\widetilde{\mathfrak m},v},N_{\widetilde{\mathfrak m},v})$ o˘ 
$N_{\widetilde{\mathfrak m},v}$ est le logarithme de la partie unipotente de la monodromie locale.  
Notons que cet opÈrateur nilpotent est dÈfini via la somme finie 
$$\ln(1-x)=- \sum_{k=1}^{d_{\widetilde{\mathfrak m},v} -1} \frac{x^k}{k},$$ 
o˘  $d_{\widetilde{\mathfrak m},v}$ est
l'ordre de nilpotence de $N_{\widetilde{\mathfrak m},v}$. En particulier 
\begin{itemize}
\item si $\rho_{\widetilde{\mathfrak m}}$ est entiËre, i.e. s'il existe un rÈseau stable
$\Gamma$ et si $l \geq d_{\widetilde{\mathfrak m},v}-1$, alors l'opÈrateur 
$N_{\widetilde{\mathfrak m},v}$ est dÈfini sur $\Gamma$
et on note $\overline N_{\widetilde{\mathfrak m},v,\Gamma}$ sa rÈduction modulo l'idÈal maximal 
de l'anneau des entiers de $\overline \Qm_l$. 
Dans la suite $l$ vÈrifiera toujours l'inÈgalitÈ $l \geq d_{\widetilde{\mathfrak m},v}-1$.

\item Si en outre la rÈduction modulo $l$
de $\rho_{\widetilde{\mathfrak m}}$ est irrÈductible alors tous
les rÈseaux stables $\Gamma$ sont homothÈtiques et on notera simplement 
$\overline N_{\widetilde{\mathfrak m},v}$.
\end{itemize}

\begin{nota} 
Tout ÈlÈment nilpotent de $GL_d(F_v)$ admet une forme de Jordan associÈe ‡ une
unique partition de $d$ donnÈe par la taille de ses blocs de Jordan. On note 
$$\underline{d_{\widetilde{\mathfrak m},v}} \quad \hbox{resp. } \underline{d_{\mathfrak m,v}}$$
la partition associÈe ‡ $N_{\widetilde{\mathfrak m},v}$ (resp. ‡ $\overline{N_{\mathfrak m,v}}$).
\end{nota}

On introduit alors les diagrammes de Ferrers ÈtiquetÈs $T_{\widetilde{\mathfrak m},v}$
et $T_{\mathfrak m,v}$ associÈs respectivement ‡  $N_{\widetilde{\mathfrak m},v}$ et 
$\overline{N_{\mathfrak m,v}}$ dont les longueurs des lignes sont les tailles, 
classÈes par ordre dÈcroissant, des blocs de Jordan de l'opÈrateur de monodromie associÈ.
En particulier $d_{\widetilde{\mathfrak m},v}$ est la longueur de la premiËre ligne.
On rappelle par ailleurs que les longueurs des colonnes de $T_{\widetilde{\mathfrak m},v}$
sont les $\dim \ker N_{\widetilde{\mathfrak m},v}^{i+1} - \dim \ker N_{\widetilde{\mathfrak m},v}^i$.
On a une formule analogue pour $T_{\mathfrak m,v}$ de sorte qu'en particulier
$$ \underline{d_{\mathfrak m,v}} \leq \underline{d_{\widetilde{\mathfrak m},v}}.$$

\rem Avec les notations de la dÈfinition \ref{defi-TY1} et la description
de la correspondance de Langlands locale, $T_{\widetilde{\mathfrak m},v}$ est
le diagramme de Ferrers $T(\Pi_{\widetilde{\mathfrak m},v})$.

\subsection{…noncÈ du thÈorËme principal}
\label{para-enonce}

¿ prÈsent on notera simplement $\Tm_I$ pour $\Tm_I(\xi)$.
ConsidÈrons un idÈal maximal $\mathfrak m$ de $\Tm_I$ qui est $\xi$-cohomologique au sens o˘ au 
moins un idÈal premier minimal $\widetilde{\mathfrak m} \subset \mathfrak m$ l'est.
Pour un tel $\widetilde{\mathfrak m}$, on a une injection 
$\Tm_I/\widetilde{\mathfrak m} \hookrightarrow \OC_{\widetilde{\mathfrak m}}$
o˘ $\OC_{\widetilde{\mathfrak m}}$ dÈsigne l'anneau des entiers de $K_{\widetilde{\mathfrak m}}$:
la composÈe de cette injection avec la rÈduction modulo l'idÈal maximal de 
$\OC_{\widetilde{\mathfrak m}}$ coÔncide alors avec la 
surjection $\Tm_I/\widetilde{\mathfrak m} \twoheadrightarrow \Tm_I/\mathfrak m$.
On peut ainsi parler de la rÈduction modulo $\mathfrak m$ des paramËtres de Satake
$S_{\widetilde{\mathfrak m}}(w)$ lesquels ne dÈpendent donc que de $\mathfrak m$ et sont donc
donnÈs par le multi-ensemble des racines du polynÙme de Hecke en $w$
$$P_{\mathfrak{m},w}(X):=\sum_{i=0}^d(-1)^i q_w^{\frac{i(i-1)}{2}} \overline{T_{w,i}} X^{d-i} \in \overline 
\Fm_l[X]$$
i.e.
$$S_{\mathfrak{m}}(w) := \bigl \{ \lambda \in \Tm_I/\mathfrak m \simeq \overline \Fm_l \hbox{ tel que }
P_{\mathfrak{m},w}(\lambda)=0 \bigr \}.$$
Ainsi tout idÈal premier $\widetilde{\mathfrak m} \subset \mathfrak m$ qui est $\xi$-cohomologique,
fournit une (ou des) reprÈsentation automorphe dont les paramËtres de Satake modulo $l$
en toute place de $\spl(I)$ sont donnÈs par $\mathfrak m$; pour deux tels idÈaux premiers
on obtient ainsi des reprÈsentations automorphes dites congruentes, au sens o˘ elles partagent
les mÍmes paramËtres de Satake en presque toutes les places.

\begin{defi} \label{defi-deter}
On dira que $\overline{N_{\mathfrak m,v}}$ est dÈtÈriorÈ relativement ‡ $\widetilde{\mathfrak m}$,
si, cf. la notation \ref{nota-partition1} et la dÈfinition \ref{defi-partition-ordre}, 
$\underline{d_{\widetilde{\mathfrak m},v}}^{(1)}$ n'est pas contenu dans $\underline{d_{\mathfrak m,v}}$.
\end{defi}

\begin{theo} \label{theo-principal}
Soient 
\begin{itemize}
\item $I \in \IC$ un sous-groupe compact ouvert tel que modulo $\varpi_v$, la composante $I_v$ de $I$ 
‡ la place $v$, cf. la formule (\ref{eq-facteur-v}), est 
un sous-groupe parahorique relativement ‡ une partition $\underline m=(m_1 \geq \cdots \geq m_r)$ 
de $d$ avec $r>1$,

\item et, cf la remarque suivant la dÈfinition \ref{nota-spl2}, $\mathfrak m$ un idÈal maximal de $\Tm_I$, 
\end{itemize}
tels que $\underline m$ soit maximal relativement au fait qu'il existe un idÈal premier minimal
$\widetilde{\mathfrak m} \subset \mathfrak m$ ainsi qu'une reprÈsentation automorphe irrÈductible 
$\Pi \in \Pi_{\widetilde{\mathfrak m}}$ possÈdant des vecteurs non nuls invariants sous $I$.
On suppose alors que:
\begin{enumerate}[1)]
\item $\overline \rho_{\mathfrak m}$, est irrÈductible

\item $l \geq d_{\widetilde{\mathfrak m},v}-1$;

\item la partition $\underline{d_{\mathfrak m,v}}$ est strictement plus petite que 
la partition $\underline m^*$ conjuguÈe ‡ $\underline m$;

\item au choix
\begin{itemize}
\item[(i)] soit $\overline{N_{\mathfrak m,v}}$ est dÈtÈriorÈ relativement ‡ $\widetilde{\mathfrak m}$,

\item[(ii)] soit la composante locale en $v$ de $\Pi_{\widetilde{\mathfrak m}}$ n'est pas de la 
forme $\chi_{v,1} \times \chi_{v,2} \times ?$ pour $\chi_{v,1}$ et $\chi_{v,2}$ des 
caractËres de $F_v^\times$ tels que $\chi_{v,2} \equiv \chi_{v,1} \nu \mod l$ o˘
$$\nu:x \in F_v^\times \mapsto q^{-\val x},$$
et $?$ une reprÈsentation quelconque.
\end{itemize}

%
%
\end{enumerate}
Alors pour toute place $w \in \spl(I)$ distincte de $v$, il existe un sous-groupe parahorique $I'_v$ (resp. $I'_w$) 
associÈe ‡ une partition $\underline m_v' >\underline m$ 
(resp. $\underline m_w'$) ainsi qu'un idÈal maximal $\xi$-cohomologique
$\mathfrak m'$ de $\Tm_{I'}$, o˘ $I':=I'_v I'_w I^{v,w}$, tel que 
$$\overline \rho_{\mathfrak m'} \simeq \overline \rho_{\mathfrak m}.$$
Autrement dit $\overline \rho_{\mathfrak m}$ provient d'une reprÈsentation automorphe de niveau $I'$.
\end{theo}

\rem Comme $l \geq d_{\widetilde{\mathfrak m},v}-1$ et $\overline \rho_{\mathfrak m}$ est irrÈductible, 
$\overline N_{\mathfrak m,v}$ est bien dÈfini indÈpendamment du rÈseau stable. 
Par maximalitÈ de $\underline m$, on a $d_{\widetilde{\mathfrak m},v}=\underline m^*$ de sorte que
l'hypothËse $\underline{d_{\mathfrak m,v}} < \underline{d_{\widetilde{\mathfrak m},v}}$
est clairement nÈcessaire pour obtenir un ÈnoncÈ de diminution du niveau. 

Le principe de la dÈmonstration
consiste ‡ calculer la cohomologie en niveau $I$ de la variÈtÈ de Shimura associÈe ‡ $G$, cf.
le \S \ref{para-KHT}, en utilisant la suite spectrale de Rapoport-Zink en la place $v$. On constate alors
\begin{itemize}
\item $\overline \rho_{\mathfrak m}$ Ètant irrÈductible, sur $\overline \Qm_l$ cette suite spectrale
dÈgÈnËre en $E_1$, tous les groupes de cohomologie Ètant concentrÈs en degrÈ mÈdian;

\item l'hypothËse $\underline{d_{\mathfrak m,v}} < \underline{d_{\widetilde{\mathfrak m},v}}$
impose que certains des termes initiaux de cette suite spectrale ont de la torsion non triviale.
\end{itemize}
Ainsi la cohomologie de toute la variÈtÈ de Shimura admet une filtration dont certains graduÈs sont 
de torsion et la partie la plus difficile de la preuve consiste, proposition \ref{prop-torsion-princ}
\begin{itemize}
\item ‡ montrer que la cohomologie elle-mÍme admet de la torsion non triviale 

\item et que celle-ci subsiste en diminuant lÈgËrement le niveau en $v$.
\end{itemize}
On utilise alors le rÈsultat principal de \cite{boyer-mrl} qui permet de relever une telle classe de torsion
en caractÈristique $0$ quitte ‡ augmenter le niveau en une place auxiliaire.

Une question naturelle est d'itÈrer ce rÈsultat afin de construire un niveau $I'=I^{v,w}I'_v I'_w$
avec $I'_w \subset I_w$ et $I'_v$ un sous-groupe parahorique associÈ ‡ une partition
$\underline m$,
de sorte qu'il existe $\widetilde{\mathfrak m}$ avec 
$$\underline{d_{\widetilde{\mathfrak m},v}}=\underline{d_{\mathfrak m,v}}=\underline m$$
et $\Pi \in \Pi_{\widetilde{\mathfrak m}}$ ayant des vecteurs non nuls invariants sous $I'$.
C'est clairement le cas si  $\underline{d_{\mathfrak m,v}}$ n'admet pas plus d'une ligne de longueur 
$1$ puisqu'alors la condition 4-ii) est toujours vÈrifiÈe. Plus gÈnÈralement on a l'ÈnoncÈ suivant.

\begin{coro} \label{coro-principal}
Supposons que l'ensemble des Ètiquettes des lignes de longueur $1$ du diagramme de Ferrers
ÈtiquetÈ associÈ ‡ $\underline{d_{\mathfrak m,v}}$, ne contient aucun sous-ensemble de la forme 
$\{ \alpha, q \alpha \}$.
Alors pour toute place $w \in \spl(I)$ distincte de $v$, il existe un sous-groupe parahorique $I'_v$ (resp. $I'_w$) associÈe ‡ la partition $\underline{d_{\mathfrak m,v}}^*$ 
(resp. $\underline m_w'$) ainsi qu'un idÈal maximal $\xi$-cohomologique
$\mathfrak m'$ de $\Tm_{I'}$, o˘ $I':=I'_v I'_w I^{v,w}$, tel que 
$$\overline \rho_{\mathfrak m'} \simeq \overline \rho_{\mathfrak m},$$
autrement dit $\overline \rho_{\mathfrak m}$ provient d'une reprÈsentation automorphe de niveau $I'$.
\end{coro}

%
%

\rem On notera que dans le cas $d=2$ si l'hypothËse 4-ii) n'est pas vÈrifiÈ alors il n'y a rien ‡ dÈmontrer
puisqu'il existe alors un $\Pi_{\widetilde{\mathfrak m}}$ non ramifiÈ en $v$.

Au \S \ref{para-dege}, on donnera par ailleurs un ÈnoncÈ de non dÈgÈnÈrescence de la monodromie, i.e. des conditions
explicites sur $\mathfrak m$ et $\widetilde{\mathfrak m}$, pour que 
$\underline{d_{\widetilde{\mathfrak m},v}}= \underline{d_{ \mathfrak m,v}}$, cf. le corollaire \ref{coro-appli}.

\section{Cohomologie des variÈtÈs de Kottwitz-Harris-Taylor}

\subsection{Rappels sur la gÈomÈtrie}
\label{para-KHT}

On reprend les notations du \S \ref{para-nota} o˘ $G$ est un groupe de similitudes sur $\Qm$ et $p=uu^c$ un nombre premier dÈcomposÈ
dans $E$ avec une place notÈe $v$ de $F$ au dessus de $u$ telle que $(B_v^{op})^\times \simeq GL_d(F_v)$.
On note alors 
$$(X_I)_{I \in \IC}$$ 
le systËme projectif des variÈtÈs de Shimura associÈ au groupe $G$ au dessus de $\spec \OC_v$ tel qu'il
est introduit dans \cite{h-t}: ces variÈtÈs de Shimura sont dites de Kottwitz-Harris-Taylor.
Ce systËme projectif est muni d'une action de $G(\Am^\oo) \times \Zm$  telle que l'action 
d'un ÈlÈment $w_v$ du groupe de Weil $W_v$ de $F_v$ est donnÈe par celle de $-\deg (w_v) \in \Zm$,
avec $\deg=\val \circ \art^{-1}$ o˘ $\art^{-1}:W_v^{ab} \simeq F_v^\times$ est
l'isomorphisme d'Artin qui envoie les Frobenius gÈomÈtriques sur les uniformisantes.

%

\begin{nota}
On note: $X_{I,s_v}$ (resp. $X_{I,\eta_v}$) la fibre spÈciale (resp. gÈnÈrique)
de $X_I$ en $v$ et $X_{I,\bar s_v}:=X_{I,s_v} \times \spec \overline \Fm_p$ la fibre spÈciale gÈomÈtrique
(resp. $X_{I,\bar \eta_v}$ la fibre gÈnÈrique gÈomÈtrique).
\end{nota}

%
%
%
%
%

Pour $I \in \IC$ tel que, cf. la dÈfinition \ref{defi-Kv}, sa composante 
$I_v$ ‡ la place $v$, est le sous-groupe parahorique standard associÈ ‡ la partition
$(m_1 \geq \cdots \geq m_r)$ de $d$. Alors le morphisme $X_{I} \longrightarrow X_{I^v}$ est le 
problËme de modules correspondant ‡ la donnÈe d'une chaÓne d'isogÈnies
$$\GC_A=\GC_0 \rightarrow \GC_1 \rightarrow \cdots \rightarrow \GC_r=\GC_A/\GC_A[\varpi_v]$$
de $\OC_v$-modules de Barsotti-Tate o˘:
\begin{itemize}
\item pour tout $i=1,\cdots, r$, l'isogÈnie
$\GC_{i-1} \longrightarrow \GC_i$ est de degrÈ $q^{m_i}$ et 

\item la composÈe de ces $r$ isogÈnies est Ègale ‡ l'application canonique 
$\GC_A \longrightarrow \GC_A/\GC_A[\varpi_v]$,
\end{itemize}
o˘ $\GC_A$ est le module de Barsotti-Tate associÈ ‡ la variÈtÈ abÈlienne universelle
sur $X_{I^v}$.

\begin{nota} \label{nota-iwahori}
Pour tout $1 \leq i \leq r$,
on note $Y_{I,i}$ le sous-schÈma fermÈ de $X_{I,\bar s_v}$ sur lequel $\GC_{i-1}
\longrightarrow \GC_i$ a un noyau connexe. Pour tout $S \subset \{ 1,\cdots, r \}$ non vide, on note
$$Y_{I,S}=\bigcap_{i \in S} Y_{I,i}, \qquad Y^0_{I,S}=Y_{I,S}-\bigcup_{S \subsetneq T} Y_{I,T}.$$
Pour tout $1 \leq m \leq r$, soit
$$Y_I^{(r)}:=\coprod_{\sharp S=r} Y_{I,S} \quad \hbox{et} \quad a_r:Y_I^{(r)} \longrightarrow Y_I$$
la projection.
\end{nota}

De la thÈorie du modËle local de Rapoport-Zink, on dÈduit la description suivante, cf. par exemple
\cite{y-t}.

\begin{prop} \label{prop-semi-stable}
Le schÈma $X_I$ est de pure dimension $d$ et a rÈduction semi-stable
sur $\OC_v$, i.e. pour tout point fermÈ $x$ de $X_{I,\bar s_v}$, il existe un voisinage
Ètale $V \longrightarrow X_I$ de $x$ et un $\OC_v$-morphisme Ètale
$$V \longrightarrow \spec \OC_v[T_1,\cdots,T_d]/(T_1 \cdots T_m-\varpi_v)$$
pour $1 \leq m \leq d$. Le schÈma $X_I$ est rÈgulier et le morphisme de restriction
du niveau $X_I \longrightarrow X_{I^v}$ est fini et plat.
Tous les $Y_{I,S}$ sont lisses sur $\spec \kappa(w)$ de pure dimension $d-\sharp S$ avec
$$X_{I,\bar s_v}=\bigcup_{i=1}^r Y_{I,i}$$
o˘ pour $i \neq j$, les schÈmas $Y_{I,i}$ et $Y_{I,j}$ n'ont pas de composante
connexe en commun. 
%
\end{prop}
%

La fibre spÈciale gÈomÈtrique $X_{I,\bar s_v}$, quel que soit le niveau $I$, admet une stratification
dite de Newton: pour tout $1 \leq h \leq d$, on note 
$$X_{I,\bar s_v}^{\geq h}, \quad \hbox{resp. } X_{I,\bar s_v}^{=h},$$
la strate fermÈe (resp. ouverte) de Newton de hauteur $h$, i.e. le sous-schÈma dont la
partie connexe du groupe de Barsotti-Tate en chacun de ses points gÈomÈtriques
est de rang $\geq h$ (resp. Ègal ‡ $h$).

\begin{nota}
On note
$$j^{\geq h}:X^{=h}_{I,\bar s_v} \hookrightarrow X^{\geq h}_{I,\bar s_v}, \qquad
i^h:X^{\geq h}_{I,\bar s_v} \hookrightarrow X_{I,\bar s_v}$$
et $j^{=h}=i^h \circ j^{\geq h}$.
\end{nota}

Lorsque le niveau $I_v$ en $v$ de $I$ est de la forme $\ker \bigl ( GL_d(\OC_v) \twoheadrightarrow
GL_d(\OC_v/\varpi_v^{m_1}) \bigr )$, pour tout $1 \leq h < d$, 
la strate de Newton $X_{I,\bar s_v}^{=h}$ est alors gÈomÈtriquement induite au sens o˘
il existe un sous-schÈma fermÈ $X_{I,\bar s_v,1_h}^{=h}$ muni d'une action de $P_{h,d-h}(\OC_v/\PC_v^{m_1})$ tel que:
$$X_{I,\bar s_v}^{=h} \simeq X_{I,\bar s_v,1_h}^{=h} 
\times_{P_{h,d-h}(\OC_v/\PC_v^{m_1})} GL_d(\OC_v/\PC_v^{m_1}).$$

\subsection{SystËmes locaux d'Harris-Taylor}

Passons provisoirement en niveau infini en $v$ et notons $X_{Iv^\oo}$ la tour associÈe: ‡ l'aide des variÈtÈs d'Igusa de 
premiËre et seconde espËce, les auteurs de \cite{h-t} p136, associent
‡ toute reprÈsentation admissible $\rho_v$  des inversibles de l'ordre maximal 
$\DC_{v,h}$ de l'algËbre ‡ division centrale $D_{v,h}$ sur $F_v$ d'invariant $1/h$,
un systËme local $\LC_{1_h}(\rho_v)$ sur $X^{=h}_{Iv^\oo,\bar s_v,1_h}$ muni 
d'une action de $P_{h,d-h}(\OC_v)$ agissant via
la projection $P_{h,d-h}(\OC_v) \longrightarrow \Zm \times GL_{d-h}(\OC_v)$. On note alors
$$\LC(\rho_v):=\LC_{1_h}(\rho_v) \times_{P_{h,d-h}(\OC_v)} GL_d(\OC_v)$$
sa version induite sur $X^{=h}_{Iv^\oo,\bar s_v}$.

\begin{nota}
Soient $\pi_v$ une reprÈsentation irrÈductible cuspidale de $GL_g(F_v)$ et pour 
$1 \leq t \leq \frac{d}{g}$, on note $\pi_v[t]_D$ la reprÈsentation de $D_{tg,v}^\times$
associÈ ‡ la reprÈsentation de Steinberg $\st_t(\pi_v)$ par la correspondance locale de
Jacquet-Langlands. La reprÈsentation $\pi_v[t]_D$ de $D_{v,tg}^\times$ fournit alors
un systËme local sur $X^{=tg}_{Iv^\oo,\bar s_v,1_{tg}}$
$$\LC(\pi_v[t]_D)_{1_{tg}}=\bigoplus_{i=1}^{e_{\pi_v}} 
\LC_{\overline \Qm_l}(\rho_{v,i})_{1_{tg}}$$
o˘ $(\pi_v[t]_D)_{|\DC_{v,h}^\times}=\bigoplus_{i=1}^{e_{\pi_v}} \rho_{v,i}$ avec $\rho_{v,i}$ irrÈductible
et muni d'une action de $P_{tg,d-tg}(F_v)$ via son quotient $GL_{d-tg} \times \Zm$.
\end{nota}

\begin{defi} \label{defi-HT}
Les systËmes locaux d'Harris-Taylor sont alors les
$$\widetilde{HT}_{1_{tg}}(\pi_v,\Pi_t):=\LC(\pi_v[t]_D)_{1_{tg}} \otimes \Pi_t \otimes \Xi^{\frac{tg-d}{2}}$$
o˘ $\Pi_t$ est une reprÈsentation quelconque de $GL_{tg}(F_v)$. 
La version induite est notÈe
$$\widetilde{HT}(\pi_v,\Pi_t):=\Bigl ( \LC(\pi_v[t]_D)_{1_{tg}} 
\otimes \Pi_t \otimes \Xi^{\frac{tg-d}{2}} \Bigr) \times_{P_{tg,d-tg}(F_v)} GL_d(F_v),$$
o˘ l'action du radical unipotent de $P_{tg,d-tg}(F_v)$ est triviale, et celle de
$$(g^{\oo,v},\left ( \begin{array}{cc} g_v^c & *†\\ 0 & g_v^{et} \end{array} \right ),\sigma_v) 
\in G(\Am^{\oo,v}) \times P_{tg,d-tg}(F_v) \times W_v$$ 
est donnÈe
\begin{itemize}
\item par celle de  $g_v^c$ sur $\Pi_t$ et  $\deg(\sigma_v) \in \Zm$ sur $ \Xi^{\frac{tg-d}{2}}$ ainsi que 

\item celle de $(g^{\oo,v},g_v^{et},\val(\det g_v^c)-\deg \sigma_v)
\in G(\Am^{\oo,v}) \times GL_{d-tg}(F_v) \times \Zm$ sur $\LC_{\overline \Qm_l}
(\pi_v[t]_D)_{1_{tg}} \otimes \Xi^{\frac{tg-d}{2}}$,
o˘ $\Xi:\frac{1}{2} \Zm \longrightarrow 
\overline \Zm_l^\times$ est dÈfini par $\Xi(\frac{1}{2})=q^{1/2}$.
\end{itemize}
On dit de l'action de $GL_{tg}(F_v)$ qu'elle est \emph{infinitÈsimale}.
\end{defi}

\rem En particulier le facteur $\Pi_t$ n'intervient pas rÈellement ni dans le calcul des faisceaux de cohomologie, par exemple
des extensions intermÈdiaires de la notation suivante, ni dans celui des groupes de cohomologie, par exemple de
$H^i_c(X_{Iv^\oo,\bar s_v,1_{tg}},\widetilde HT_{1_{tg}}(\pi_v,\Pi_t))$. PrÈcisÈment, 
d'aprËs la description des actions rappelÈe ci-avant,
en tant $\Tm_I \times GL_{tg}(F_v) \times GL_{d-tg}(F_v) \times \Zm$-module, un tel groupe de cohomologie 
s'Ècrit comme une extension de modules irrÈductibles de la forme 
$M \otimes \big (\Pi_t \otimes \chi \circ \val \circ \det \big ) \otimes \Pi_v \otimes \chi^{-1}$, o˘ $M$ (resp. $\Pi_v$)
est un $\Tm_I$-module (resp. une reprÈsentation de $GL_{d-tg}(F_v)$) irrÈductible, et $\chi:\Zm \rightarrow \overline \Qm_l^\times$.

\begin{notas} On pose
$$HT(\pi_v,\Pi_t):=\widetilde{HT}(\pi_v,\Pi_t)[d-tg],$$
et le faisceau pervers d'Harris-Taylor associÈ est
$$P(t,\pi_v):= \lexp p j^{=tg}_{!*} HT(\pi_v,\st_t(\pi_v)) \otimes \Lm(\pi_v),$$
o˘ $\Lm^\vee$ dÈsigne la correspondance locale de Langlands.
\end{notas}

D'aprËs cf. \cite{boyer-invent2} proposition 4.3.1 complÈtÈe par le corollaire 5.4.1, on a l'ÈgalitÈ
suivante dans le groupe de Grothendieck des faisceaux pervers Èquivariants
\begin{multline} \label{eq-egalite}
j^{=tg}_! HT(\pi_v,\Pi_t)=\lexp p j^{=tg}_{!*} HT(\pi_v,\Pi_t) \\
+\sum_{k=1}^{\lfloor \frac{d}{g} \rfloor -t}
\lexp p j^{\geq (t+k)}_{!*}
HT(\pi_v,\Pi_{t} \{ \frac{-k}{2} \} \times \st_k(\pi_v\{\frac{t}{2}\} )) (k/2)
\end{multline}

On rappelle que $\pi'_v$ est inertiellement Èquivalente ‡ $\pi_v$ si et seulement
s'il existe un caractËre $\zeta: \Zm \longrightarrow  \overline \Qm_l^\times$ tel que 
$\pi'_v \simeq \pi_v \otimes (\zeta \circ \val \circ \det)$.
Les faisceaux pervers $P(t,\pi_v)$ ne dÈpendent que de la classe d'Èquivalence inertielle de $\pi_v$
et sont de la forme 
$$P(t,\pi_v)=e_{\pi_v} \PC(t,\pi_v)$$ 
o˘ $\PC(t,\pi_v)$ est un faisceau pervers irrÈductible.

\begin{nota} Pour $I \in \IC$, on notera $\PC_I(t,\pi_v):=\PC(t,\pi_v)^{I_v}$ le faisceau pervers
d'Harris-Taylor sur $X_{I,\bar s_v}$ et on ajoutera plus gÈnÈralement un indice $I$ pour
les $HT(\pi_v,\Pi_t)$ lorsqu'on les considËre ‡ niveau fini $I$.
\end{nota}

\rem Lorsque $I_v$ est un sous-groupe parahorique, le faisceau pervers $P_I(t,\pi_v)$
est nul si $\pi_v$ n'est pas un caractËre.

Le rÈsultat principal de \cite{boyer-invent2} sur les faisceaux de cohomologies des faisceaux 
pervers d'Harris-Taylor, dont on pourra trouver une preuve simplifiÈe dans \cite{boyer-FT} 
peut s'Ècrire comme suit sous la forme d'une rÈsolution o˘ on a posÈ $s=\lfloor \frac{d}{g} \rfloor$:
\addtocounter{smfthm}{1}
\begin{multline} \label{eq-resolution0} 
0 \rightarrow j_!^{=sg} HT_{1_{tg}}(\pi_v,\Pi_t \{ \frac{t-s}{2} \} \otimes \speh_{s-t}(\pi_v\{\frac{t}{2} \})) 
\otimes \Xi^{\frac{s-t}{2}} \longrightarrow \cdots \\
\cdots \longrightarrow j_!^{=t+1} HT_{1_{tg}}(\pi_v,\Pi_t\{-1/2 \} \otimes \pi_v\{ \frac{t}{2} \}) \otimes 
\Xi^{\frac{1}{2}}  \\
\longrightarrow j_!^{=t} HT_{1_{tg}}(\pi_v,\Pi_t) \longrightarrow \lexp p j_{!*}^{=t} 
HT_{1_{tg}}(\pi_v,\Pi_t) \rightarrow 0,
\end{multline}
o˘ pour tout $tg \leq h \leq d$, $\Pi_t$ (resp. $\Pi_{h-tg}$) une reprÈsentation de $GL_{tg}(F_v)$
(resp. de $GL_{h-tg}(F_v)$), on a notÈ
$$HT_{1_{tg}}(\pi_v,\Pi_t \otimes \Pi_{h-tg}):=HT_{1_h}(\pi_v,\Pi_t \otimes \Pi_{h-tg}) 
\times_{P_{tg,h-tg,d-h}(F_v)} P_{tg,d-tg}(F_v).$$


Pour $\chi_v$ une reprÈsentation cuspidale de $GL_1(F_v)$, i.e. un caractËre de $F_v^\times$,
le $\overline \Zm_l$-systËme local $\LC_{1_h}(\chi_v)$ est isomorphe ‡ $\overline \Zm_l$ muni
de l'action du groupe fondamental $\Pi_1(X^{=h}_{I,\bar s_v})$ de $X^{=h}_{I,\bar s_v}$
qui se factorise par son quotient $\Pi_1(X^{=h}_{I,\bar s_v}) \twoheadrightarrow \DC_{v,h}^\times$
o˘ l'action de $\DC_{v,h}^\times$
est donnÈe par le caractËre $\chi_v$. En remarquant, cf. \cite{boyer-duke} lemme 3.0.2, 
que l'adhÈrence $X^{\geq h}_{I,\bar s_v}$ de $X^{=h}_{I,\bar s_v}$ est lisse, on en dÈduit que
$\overline \Zm_l [d-h]$ est un faisceau pervers sur $X^{\geq h}_{I,\bar s_v}$ qui s'identifie,
avec l'action de $\Pi_1(X_{I,\bar s_v}^{=h})$ comme ci-avant, alors
aux deux extensions intermÈdaires 
\addtocounter{smfthm}{1}
\begin{equation} \label{eq-p+p}
\lexp p j^{\geq h}_{!*} \LC_{1_h}(\chi_v)[d-h] \simeq 
\lexp {p+} j^{\geq h}_{!*} \LC_{1_h}(\chi_v)[d-h].
\end{equation}

\rem D'aprËs le rÈsultat principal de \cite{boyer-duke}, la rÈsolution prÈcÈdente 
(\ref{eq-resolution0}) est encore valide sur $\overline \Zm_l$. En revanche
l'isomorphisme (\ref{eq-p+p}) n'est valable que lorsque la rÈduction modulo $l$ de $\pi_v$
est encore supercuspidale et pas seulement cuspidale. Dans ce texte nous ne considÈrerons
que le cas $g=1$, i.e. les $\pi_v$ qui sont des caractËres $\chi_v$.

\begin{nota} \label{nota-uspilon}
On notera $\Upsilon$ l'ensemble des classes d'Èquivalence inertielle des caractËres de $F_v^\times$.
\end{nota}

\subsection{RelËvement des classes de cohomologie de torsion}
\label{para-xi}

Fixons un plongement $\sigma_0:E \hookrightarrow
\overline{\Qm}_l$ et notons $\Phi$ l'ensemble des plongements $\sigma:F \hookrightarrow
\overline \Qm_l$ dont la restriction ‡ $E$ est $\sigma_0$. 
On rappelle alors qu'il existe une bijection explicite entre les reprÈsentations algÈbriques irrÈductibles 
$\xi$ de $G$ sur $\overline \Qm_l$ et les $(\sharp \Phi+1)$-uplets
$$\bigl ( a_0, (\overrightarrow{a_\sigma})_{\sigma \in \Phi} \bigr )$$
o˘ $a_0 \in \Zm$ et pour tout $\sigma \in \Phi$, on a $\overrightarrow{a_\sigma}=
(a_{\sigma,1} \leq \cdots \leq a_{\sigma,d} )$.
Il existe alors une extension finie $K$ de $\Qm_l$ telle que 
la reprÈsentation $\iota^{-1} \circ \xi$ de plus haut poids
$\bigl ( a_0, (\overrightarrow{a_\sigma})_{\sigma \in \Phi} \bigr )$,
est dÈfinie sur $K$. On note $W_{\xi,K}$ l'espace de cette reprÈsentation et $W_{\xi,\OC}$
un rÈseau stable sous l'action du sous-groupe compact maximal $G(\Zm_l)$ et 
o˘ $\OC$ dÈsigne l'anneau des entiers de $K$.

\rem Si on suppose que $\xi$ est $l$-petit, i.e. que pour tout $\sigma \in \Phi$ 
$a_{\sigma,d}-a_{\sigma,1} < l$,
alors un tel rÈseau stable est unique ‡ homothÈtie prËs.

Notons $\lambda$ une uniformisante de $\OC$ et soit
pour $n \geq 1$, un sous-groupe distinguÈ $I_n \in \IC$ de $I \in \IC$,
compact ouvert agissant trivialement sur $W_{\xi,\OC/\lambda^n}:=W_{\xi,\OC} 
\otimes_{\OC} \OC/\lambda^n$. On note alors $V_{\xi,\OC/\lambda^n}$ le faisceau sur
$X_{I}$ dont les sections sur un ouvert Ètale $T \longrightarrow X_{I}$ sont les fonctions
$$f:\pi_0 \big ( X_{I_n} \times_{X_I} T \bigr ) \longrightarrow W_{\xi,\OC/\lambda^n}$$
telles que pour tout $k \in I$ et $C \in \pi_0 \big ( X_{I_n} \times_{X_I} T \bigr )$, on a
la relation $f(Ck)=k^{-1} f(C)$.

\begin{notas} On note
$$V_{\xi,\OC}=\lim_{\atop{\leftarrow}{n}} V_{\xi,\OC/\lambda^n} \hbox{ et }
V_{\xi,K}=V_{\xi,\OC} \otimes_{\OC} K.$$
On utilisera aussi la notation $V_{\xi,\overline \Zm_l}$ et $V_{\xi,\overline \Qm_l}$ pour les versions
sur $\overline \Zm_l$ et $\overline \Qm_l$ respectivement ainsi que 
$$HT_\xi (\pi_v,\Pi_t):=HT(\pi_v,\Pi_t) \otimes V_{_xi,\overline \Qm_l}.$$
\end{notas}

\rem 
La reprÈsentation $\xi$ est dite \emph{rÈguliËre} si son paramËtre
$\bigl ( a_0, (\overrightarrow{a_\sigma})_{\sigma \in \Phi} \bigr )$ est tel que pour 
tout $\sigma \in \Phi$, on a $a_{\sigma,1} < \cdots < a_{\sigma,d}$. 

On rappelle le rÈsultat principal de \cite{boyer-mrl} qui permet de relever en caractÈristique nulle les classes de torsion.

\begin{theo} (cf. \cite{boyer-mrl} corollaire 2.9)
Soit $i$ tel que le sous-module de torsion de 
$H^i(X_{I,\bar \eta},V_{\xi,\overline \Zm_l})_{\mathfrak m}$ est non nul. Alors pour tout $v \in \spl(I)$,
il existe une reprÈsentation irrÈductible $\xi$-cohomologique $\Pi(v)$ non ramifiÈe en toute
place $w \neq v$ ne divisant pas $I$ et dont les paramËtres de Satake modulo $l$ en $w$
sont donnÈs par $S_{\mathfrak m}(w)$. 
\end{theo}

\rem La composante en $v$ de $\Pi(v)$ est ramifiÈe et d'aprËs loc. cit. en utilisant
le lemme \ref{lem-secherre}, possËde des
vecteurs non nuls invariants sous un certain sous-groupe parahorique associÈ
‡ une partition de la forme $(m \geq 1 \geq 1 \geq \cdots \geq 1)$.

D'aprËs le thÈorËme prÈcÈdent, pour prouver \ref{theo-principal} il suffit alors de montrer, 
cf. la proposition \ref{prop-torsion-princ}, 
qu'il existe un niveau $I'$ de la forme $I'=I^vI'_v$ o˘ $I'_v$ est un sous-groupe parahorique
contenant strictement $I_v$ tel que la localisation en 
$\mathfrak m$ de la cohomologie de $X_{I'}$ ‡ coefficients dans $V_\xi$ est non nulle et donc,
par maximalitÈ de $\underline m$ nÈcessairement de torsion.

\section{Preuve du thÈorËme principal}

\subsection{Suite spectrale de Rapoport-Zink}

On considËre ‡ prÈsent un niveau $I \in \IC$ tel que la composante $I_v$ de $I$ 
‡ la place $v$ est le sous-groupe parahorique standard $\Iw_v(\underline m)$ associÈ ‡ la partition
$\underline m=(m_1 \geq m_2 \geq \cdots \geq m_r)$ de $d$. Avec les notations de \ref{nota-iwahori},
la variÈtÈ de Shimura $X_I$ admet une rÈduction semi-stable ‡ la place $v$ ce qui permet de
reprendre les constructions de Rapoport-Zink, cf. par exemple \cite{ill} \S 3. 

\begin{nota}
On note $R\Psi_{I;v}(\overline \Zm_l)$
le complexe des cycles proches sur $X_{I,\bar s_v}$.
\end{nota}

Rapoport et Zink construisent en particulier un bicomplexe $\AC$ ainsi qu'un isomorphisme de complexes
$$R\Psi_{I,v}(\overline \Zm_l) \simeq s(\AC)$$
o˘ $s(\AC)$ est le complexe simple associÈ ‡ $\AC$ ainsi qu'un morphisme
$$\nu: \AC \longrightarrow \AC[-1,1](-1)$$
qui via l'isomorphisme prÈcÈdent fournit
$$(T-1) \otimes T^\vee: R\Psi_{I,v}(\overline \Zm_l) \longrightarrow R\Psi_{I,v}(\overline \Zm_l)(-1).$$
Le bicomplexe $\AC$ est ensuite muni d'une filtration croissante $W_\bullet \AC$ de sorte que
les graduÈs correspondant $\gr^W_\bullet s(\AC)$ sont les $\gr^W_\bullet s(\BC)$ o˘
$\BC$ est le bicomplexe ‡ diffÈrentielles nulles, cf. les notations \ref{nota-iwahori}
$$\begin{array}{llcc} 
a_{r,*} \overline \Zm_l \\
a_{r-1,*} \overline \Zm_l & a_{r,*} \overline \Zm_l (-1) \\
\cdots \\
a_{1,*} \overline \Zm_l & a_{2,*} \overline \Zm_l (-1) & \cdots & a_{r,*} \overline \Zm_l (-r+1)
\end{array}$$
o˘ le coin en bas ‡ gauche correspond ‡ $(0,0)$
et o˘ $W_r \BC$ est obtenu en appliquant le foncteur de troncation canonique $\tau_{\leq r+q}$
‡ la $q$-iËme ligne de $\BC$. Ainsi le graduÈ $\gr^W_rR\Psi_{I,v}(\overline \Zm_l)$ a pour faisceaux
de cohomologie
\addtocounter{smfthm}{1}
\begin{equation} \label{eq-hipsi}
\bigl ( \hi^i \gr^W_r R\Psi_{I,v}(\overline \Zm_l) \bigr )_{i \geq 0}=\Bigl ( \overbrace{0, \cdots, 0}^r, 
a_{|r|+1,*} \overline \Zm_l (-|r|), a_{|r|+3,*} \overline \Zm_l (-|r|-1), \cdots \Bigr ).
\end{equation}
Rapoport et Zink montrent en outre que l'opÈrateur $(T-1) \otimes T^\vee$ dÈfini plus haut, induit
un isomorphisme
\addtocounter{smfthm}{1}
\begin{equation} \label{eq-RZ2}
((T-1) \otimes T^\vee)^r: \gr^W_rR\Psi_{I,v}(\overline \Zm_l) \overset{\sim}{\longrightarrow} \gr^W_{-r} R\Psi_{I,v}(\overline \Zm_l)(-r).
\end{equation}
Sur $\overline \Qm_l$, $(T-1) \otimes T^\vee$ est nilpotent ce qui permet de dÈfinir
$$N=\log T \otimes T^\vee:R\Psi_{I,v}(\overline \Qm_l)[d-1] \longrightarrow R\Psi_{I,v}(\overline \Qm_l)[d-1](-1),$$
lequel correspond alors ‡ l'opÈrateur de monodromie usuel.
Ainsi les $\gr^W_\bullet s(\AC) \otimes_{\overline \Zm_l} \overline \Qm_l$ sont les graduÈs
de la filtration de monodromie de $R\Psi_{I,v}(\overline \Qm_l)$.

\rem Les $\gr^W_r R\Psi_{I,v}(\overline \Qm_l)$ sont dÈcrits explicitement en tout niveau 
dans \cite{boyer-invent2}.

Rappelons, cf. \cite{boyer-torsion} \S 1.4, que $\DC:=D^b_c(X_{I,\bar s_v},\overline \Zm_l)$ est muni
de deux structures perverses notÈes $p$ et $p+$
$$\begin{array}{l}
A \in \lexp p \DC^{\leq 0}
\Leftrightarrow \forall x \in X,~\hi^k i_x^* A=0,~\forall k >- \dim \overline{\{ x \} } \\
A \in \lexp p \DC^{\geq 0} \Leftrightarrow \forall x \in X,~\hi^k i_x^! A=0,~\forall k <- \dim \overline{\{ x \} }
\end{array}$$
o˘ $i_x:\spec \kappa(x) \hookrightarrow X_{I,\bar s_v}$, et
$$\begin{array}{l}
A \in \lexp {p+} \DC^{\leq 0} \Leftrightarrow \forall x \in X,
\left \{ \begin{array}{ll} \hi^i i_x^* A=0, & \forall i >- \dim \overline{\{ x \} } +1 \\
\hi^{-\dim \overline{\{ x \} } +1} i_x^* A & \hbox{de torsion} \end{array} \right. \\
A \in \lexp {p+} \DC^{\geq 0} \Leftrightarrow \forall x \in X,
\left \{ \begin{array}{ll} \hi^i i_x^! A=0, & \forall i <- \dim \overline{\{ x \} } \\
\hi^{-\dim \overline{\{ x \} }} i_x^! A & \hbox{libre} \end{array} \right.
\end{array}$$

\begin{nota} On introduit le faisceau pervers
$$\Psi_{I,v,\overline \Zm_l}:=R\Psi_{I,v}(\overline \Zm_l)[d-1](\frac{d-1}{2})$$
qui est autodual et pervers pour les deux $t$-structures $p$ et $p+$.
On notera aussi 
$$\gr^W_r \Psi_{I,v,\overline \Zm_l}:=\gr^W_r R\Psi_{I,v}(\overline \Zm_l)[d-1](\frac{d-1}{2}).$$
\end{nota}

\begin{lemm} Les $\gr^W_r \Psi_{I,v,\overline \Zm_l}$ sont pervers pour les deux $t$-structures $p$ et $p+$.
\end{lemm}

\begin{proof}
De la description donnÈe plus haut de $\gr_r^W s(\BC)$, lequel ‡ un dÈcalage prËs correspond ‡ 
$\gr^W_r \Psi_{\overline \Zm_l}$, et donc de (\ref{eq-hipsi}), on en dÈduit que ce dernier appartient ‡ 
$\lexp p \DC^{\leq 0}(X_{I,\bar s_v},\overline \Zm_l) \subset 
\lexp {p+} \DC^{\leq 0}(X_{I,\bar s_v},\overline \Zm_l)$. De la lissitÈ des $Y_{I,S}$ et donc des
$Y_I^{(r)}$, on obtient de mÍme, aprËs application du foncteur de dualitÈ de Grothendieck-Verdier, que
$\gr^W_r \Psi_{I,v,\overline \Zm_l} \in \lexp {p+} \DC^{\geq 0}(X_{I,\bar s_v},\overline \Zm_l) \subset
\lexp p \DC^{\geq 0}(X_{I,\bar s_v},\overline \Zm_l)$,
d'o˘ le rÈsultat.

\end{proof}

Rappelons que $\gr^W_r \Psi_{I,v,\overline \Qm_l}$ Ètant pur, il est semi-simple et s'Ècrit 
d'aprËs \cite{boyer-invent2} thÈorËme 2.2.4
$$\gr^W_r \Psi_{I,v,\overline \Qm_l}=\bigoplus_{\genfrac{}{}{0pt}{}{1 \leq h \leq d}{h \equiv r+1 \mod 2}}
\bigoplus_{\chi_v \in \Upsilon} \PC_I(h,\chi_v)(\frac{r}{2}),$$
o˘, cf. la notation \ref{nota-uspilon}, $\Upsilon$ dÈsigne l'ensemble des classes d'Èquivalence inertielle des caractËres de $F_v^\times$.

\begin{lemm} \label{lem-decompo-gr}
Sur $\overline \Zm_l$, on a une dÈcomposition
$$\gr^W_r \Psi_{I,v,\overline \Zm_l}=\bigoplus_{\genfrac{}{}{0pt}{}{1 \leq h \leq d}{h \equiv r+1 \mod 2}}
\gr^W_{r,h} \Psi_{I,v,\overline \Zm_l},$$
o˘ $\gr^W_{r,h} \Psi_{I,v,\overline \Zm_l} \otimes_{\overline \Zm_l} \overline \Qm_l \simeq 
\bigoplus_{\chi_v \in \Upsilon} \PC_I(h,\chi_v)(\frac{r}{2}).$
\end{lemm}

\begin{proof}
Le rÈsultat dÈcoule d'aprËs (\ref{eq-p+p}) du fait dÈtaillÈ ci-dessous.
ConsidÈrons une extension
$$0 \rightarrow A_1 \longrightarrow A \longrightarrow A_2 \rightarrow 0$$
o˘ $A_1$ et $A_2$ sont des $p$-extensions intermÈdiaires de systËmes locaux sur respectivement
$X^{=h_1}_{I,\bar s_v}$ et $X^{=h_2}_{I,\bar s_v}$ avec $h_1 > h_2$ et telle que
$A \otimes_{\overline \Zm_l} \overline \Qm_l$ est scindÈe. Soit alors $A'_2$ le tirÈ en arriËre
$$\xymatrix{
A'_2 \ar@{^{(}-->}[r] \ar@{^{(}-->}[d] & A \ar@{^{(}->}[d] \\
A_2\otimes_{\overline \Zm_l} \overline \Qm_l  \ar@{^{(}->}[r] &  A  \otimes_{\overline \Zm_l} \overline \Qm_l 
}$$
de sorte que
$$\xymatrix{
 & A_1 \ar@{^{(}->}[d] \ar@{=}[r] & A_1 \ar@{^{(}->}[d] \\
 A'_2 \ar@{^{(}->}[r] \ar@{=}[d] & A \ar@{->>}[r] \ar@{->>}[d] & A'_1 \ar@{->>}[d]  \\
 A'_2 \ar@{^{(}->}[r] & A_2 \ar@{->>}[r] & T  \\
}
$$
Comme $A_2$ est une $p$-extension intermÈdiaire, si $T$ Ètait non nul, sa restriction ‡ 
$X^{=h_2}_{I,\bar s_v}$ serait non nulle ce qui ne se peut pas puisque cette strate
n'intersecte pas $X^{\geq h_1}_{I,\bar s_v}$. Ainsi donc $A$ est scindÈe.

\end{proof}

On fixe une fois pour toute une ÈnumÈration de $\Upsilon=\{ \chi_{v,1},\chi_{v,2},\cdots \}$ et on
considËre le tirÈ en arriËre
$$\xymatrix{
\PC_{I,\Gamma_r}(\chi_{v,1},h)(\frac{r}{2}) \ar@{^{(}-->}[r] \ar@{^{(}-->}[d] & 
\gr^W_r \Psi_{I,v,\overline \Zm_l} \ar@{^{(}->}[d] \\
\PC_{I}(\chi_{v,1},h)(\frac{r}{2})  \ar@{^{(}->}[r] &  \gr^W_r \Psi_{I,v,\overline \Qm_l},
}$$
Soit alors le quotient $\gr^W_{r,\geq 2} \Psi_{I,v,\overline \Zm_l}:=
\gr^W_r \Psi_{I,v,\overline \Zm_l} / \PC_{I,\Gamma_r}(\chi_{v,1},h)(\frac{r}{2})$.
On procËde alors comme prÈcÈdemment en considÈrant le tirÈ en arriËre
$$\xymatrix{
\PC_{I,\Gamma_r}(\chi_{v,2},h)(\frac{r}{2}) \ar@{^{(}-->}[r] \ar@{^{(}-->}[d] & 
\gr^W_{r,\geq 2} \Psi_{I,v,\overline \Zm_l} \ar@{^{(}->}[d] \\
\PC_{I}(\chi_{v,2},h)(\frac{r}{2})  \ar@{^{(}->}[r] &  \gr^W_{r,\geq 2} \Psi_{I,v,\overline \Qm_l},
}$$
et ainsi de suite de faÁon ‡ obtenir des structures entiËres $\PC_{I,\Gamma_r}(\chi_v,h)(\frac{r}{2})$
pour tout $h \equiv r+1 \mod 2$ et $1 \leq h \leq d$. 

\begin{lemm} Les structures entiËres $\PC_{I,\Gamma_r}(\chi_v,h)(\frac{r}{2})$ ne dÈpendent pas
de $r$ . Par ailleurs on a des isomorphismes
\addtocounter{smfthm}{1}
\begin{equation} \label{eq-RZ2b}
(T-1) \otimes T^\vee: \PC_{I,\Gamma_r}(\chi_v,h) (\frac{r}{2}) \overset{\sim}{\longrightarrow} 
\PC_{I,\Gamma_{r-2}}(\chi_v,h) (\frac{r-2}{2})
\end{equation}
pour tout $2 \leq h \leq d$, $h \equiv r-1 \mod 2$ et $3-h \leq r \leq h-1$.
\end{lemm}

\begin{proof}
D'aprËs (\ref{eq-RZ2}) et la dÈcomposition du lemme \ref{lem-decompo-gr}, 
$(T-1) \otimes T^\vee$ induit des isomorphismes
$$(T-1) \otimes T^\vee: \gr^W_{r,h} \Psi_{I,v,\overline \Zm_l} \overset{\sim}{\longrightarrow}
\gr^W_{r-2,h} \Psi_{I,v,\overline \Zm_l}(-1)$$
pour tout $2 \leq h \leq d$, $h \equiv r-1 \mod 2$ et $3-h \leq r \leq h-1$. Le rÈsultat en dÈcoule alors
puisqu'on utilise, pour tous ces $r$, la mÍme numÈrotation de $\Upsilon$ pour construire les 
structures entiËres $\PC_{I,\Gamma_r}(\chi_{v,i},h)(\frac{r}{2})$.

\end{proof}

\begin{nota} On notera alors plus simplement $\PC_{I,\Gamma}(\chi_v,t)$ la structure entiËre
de $\PC_I(\chi_v,t)$ fournie par $\Psi_{I,v,\overline \Zm_l}$ et le choix de l'ÈnumÈration de $\Upsilon$.
\end{nota}

\rem On ne cherche pas ici ‡ prÈciser de quelle structure entiËre il s'agit. Lorsque le niveau en $v$ 
est grand, on peut montrer que plusieurs telles structures coexistent pour les $\PC_I(\pi_v,t)$ lorsqu'on filtre $\Psi_{I,v,\overline \Zm_l}$.

La suite spectrale dite de Rapoport-Zink associÈe
\addtocounter{smfthm}{1}
\begin{equation} \label{eq-RZ}
E_1^{p,q}=H^{p+q}(X_{I,\bar s_v},\gr^W_{-p} R\Psi_{I,v} (\overline \Zm_l) \otimes V_{\xi,\overline \Zm_l}) 
\Rightarrow H^{p+q}(X_{I,\bar \eta},V_{\xi,\overline \Zm_l}),
\end{equation}
peut alors se raffiner  en utilisant les $\PC_{I,\Gamma}(\chi_v,h)(\frac{r}{2})$, ou comme dans
\cite{ill}, se dÈcrire ‡ l'aide des $Y_{I,S}$:
$$E_1^{p,q}=\bigoplus_{i \geq \max \{ 0, - p \} } \bigoplus_{\sharp S=p+2i+1} H^{q-2i} (Y_{I,S},
V_{\xi,\overline \Zm_l}(-i)).$$

\begin{prop} \label{prop-cohoQ}
Soit $I \in \IC$ avec donc $I_v$ un sous-groupe parahorique.
Soit $\widetilde{\mathfrak m}$ un idÈal premier de $\Tm_I$ tel qu'il existe $r$ et $i \neq 0$
avec $H^i(X_{I,\bar s_v},\gr^W_r \Psi_{I,v,\overline \Qm_l}\otimes V_{\xi,\overline \Qm_l})_{\widetilde{\mathfrak m}} \neq 0$.
Alors la reprÈsentation galoisienne $\rho_{\widetilde{\mathfrak m}}$ associÈe est rÈductible.
\end{prop}

\begin{proof}
Le thÈorËme 2.2.4 de \cite{boyer-invent2} dÈcrit les graduÈs\footnote{On montre dans loc. cit.
que les filtrations de monodromie et de poids de $\Psi_{I,v,\overline \Qm_l}$ coÔncident ‡ un dÈcalage prËs.}
$\gr^W_r \Psi_{I,v,\overline \Qm_l}$ en termes des faisceaux pervers d'Harris-Taylor lesquels
sont indexÈs par les reprÈsentations irrÈductibles cuspidales d'un $GL_g(F_v)$ pour $g$
variant de $1$ ‡ $d$. En niveau parahorique ‡ la place $v$, seules les cuspidales (caractËres)
pour $g=1$ contribuent.

Les groupes de cohomologie des faisceaux pervers d'Harris-Taylor sont explicitÈs au \S 3 de
\cite{boyer-compositio}. 
Pour ce faire on dÈcrit la partie $\Pi^\oo$-isotypique de ces groupes de cohomologie
pour $\Pi$ une reprÈsentation automorphe cohomologique.
On note alors que pour avoir de la cohomologie $\Pi^\oo$-isotypique en dehors du degrÈ mÈdian il
faut que la composante locale $\Pi_v$ en $v$ soit de la forme $\speh_s(\pi_v)$ pour $\pi_v$ une
reprÈsentation tempÈrÈe, auquel cas $\Pi^\oo$ est de la forme $\speh_s(\pi)$ pour $\pi$ cuspidale,
ce qui en termes galoisiens signifie que la reprÈsentation galoisienne associÈe ‡ $\Pi$ par la 
correspondance de Langlands globale s'Ècrit 
$\rho |-|^{\frac{1-s}{2}} \oplus \cdots \oplus \rho |-|^{\frac{s-1}{2}}$ o˘ $\rho$ est la reprÈsentation
galoisienne associÈ ‡ $\pi$ par la correspondance de Langlands globale.

\end{proof}

\begin{prop} \label{prop-libre-d}
On suppose que $\overline \rho_{\mathfrak m}$ est irrÈductible et on choisit une partition
$\underline m=(m_1 \geq \cdots \geq m_r)$ de $d$ maximale de sorte qu'il existe 
\begin{itemize}
\item $I \in \IC$ avec $I_v=\Iw_v(\underline m)$ un sous-groupe parahorique associÈ ‡ $\underline m$ et

\item un idÈal premier $\widetilde{\mathfrak m} \subset \mathfrak m$ tel que
$H^0(X_{I,\bar s_v},\Psi_{I,v,\overline \Qm_l} \otimes V_{\xi,\overline \Qm_l})_{\widetilde{\mathfrak m}}$ est non nul.
\end{itemize} 
Si en outre, tous les $H^i(X_{I,\bar s_v}, \gr^W_r \Psi_{I,v,\overline \Zm_l} \otimes V_{\xi,\overline \Zm_l})_{\mathfrak m}$ 
sont sans torsion, alors la partition $\underline{d_{\mathfrak m,v}}$ associÈ ‡ l'opÈrateur de 
monodromie est Ègale ‡ celle $\underline{d_{\widetilde{\mathfrak m},v}}$.
\end{prop}

\begin{proof}
D'aprËs la proposition prÈcÈdente si $\overline \rho_{\mathfrak m}$ est irrÈductible alors les
$H^i(X_{I,\bar s_v}, \gr^W_r \Psi_{I,v,\overline \Qm_l}\otimes V_{\xi,\overline \Qm_l})_{\mathfrak m}$ sont nuls pour tout $i \neq 0$ et
la suite spectrale (\ref{eq-RZ}) de Rapoport-Zink dÈgÈnËre en $E_1$. Par maximalitÈ de $\underline m$,
les idÈaux premiers $\widetilde{\mathfrak m} \subset \mathfrak m$ tels qu'il existe 
$\Pi \in \Pi_{\widetilde{\mathfrak m}}$ contribuant ‡ 
$H^0(X_{I,\bar s_v},\Psi_{I,v,\overline \Zm_l}\otimes V_{\xi,\overline \Zm_l})_{\mathfrak m} \otimes_{\overline \Zm_l} \overline \Qm_l$
sont tels que, d'aprËs le lemme \ref{lem-secherre} la composante locale $\Pi_v$ est de la forme
$\st_{t_1}(\chi_{v,1}) \times \cdots \times \st_{t_{m_1}}(\chi_{v,m_1})$
o˘ $(t_1\geq \cdots \geq t_{m_1})$ est la partition conjuguÈe ‡ $\underline m$ et o˘ les
$\chi_{v,i}$ sont des caractËres de $F_v^\times$. En particulier tous les
$\rho_{\widetilde{\mathfrak m}}$ fournissent la mÍme partition 
$\underline{d_{\widetilde{\mathfrak m},v}}=(t_1\geq \cdots \geq t_{m_1})$.

On peut aussi bien entendu retrouver la partition $(t_1\geq \cdots \geq t_{m_1})$ ‡ l'aide de
la filtration de monodromie et plus particuliËrement ‡ partir de la dimension de ses groupes de cohomologie .
En effet pour $i \geq 0$ et $\Pi$ une reprÈsentation automorphe irrÈductible $\xi$-cohomologique de composante locale en $v$ de la forme
$\st_{t_1}(\chi_{v,1}) \times \cdots \times \st_{t_{m_1}}(\chi_{v,m_1})$, sa contribution 
$\Bigl [ H^0(X_{I,\bar s_v}, \gr^W_{i} \Psi_{I,v,\overline \Zm_l} \otimes V_{\xi,\overline \Zm_l})_{\mathfrak m}
\otimes_{\overline \Zm_l} \overline \Qm_l \Bigr ] \{ \Pi \}$
‡ la $ \overline{\mathbb Q}_l$-cohomologie de $\gr^W_{i} \Psi_{I,v,\overline \Zm_l} \otimes V_{\xi,\overline \Zm_l}$ est
Ègal ‡ une constante $e_{\mathfrak m,I}(\Pi)$ multipliÈe par le cardinal de l'ensemble suivant
$$\bigl \{ k:~t_k \geq i+1 \hbox{ et } t_k \equiv i+1 \mod 2 \bigr \},$$
o˘ $e_{\mathfrak m,I}(\Pi)$ est, pour une reprÈsentation ayant ses paramËtres de Satake modulo $l$ donnÈs par $ \mathfrak m$, 
essentiellement donnÈe par la dimension de l'espace des invariants $(\Pi^{\oo})^I$ multipliÈe par une constante indÈpendante de $\Pi$,
cf. la dÈfinition 3.3.3 de \cite{boyer-compositio}.

Ainsi comme tous les $\Pi$ tels que $e_{\mathfrak m,I}(\Pi) \neq 0$ ont une composante locale en $v$ de la. forme
$\st_{t_1}(\chi_{v,1}) \times \cdots \times \st_{t_{m_1}}(\chi_{v,m_1})$ pour la mÍme partition $(t_1 \geq \cdots \geq t_{m_1})$, on
en dÈduit qu'il existe une constante $e$ telle que le nombre de lignes
de taille $i$ dans le diagramme de Ferrers de $\underline{d_{\widetilde{\mathfrak m},v}}$
multipliÈ par $e$ est Ègal ‡ 
$$\dim_{\overline \Qm_l} H^0(X_{I,\bar s_v}, \gr^W_{i-1} \Psi_{I,v,\overline \Zm_l} \otimes V_{\xi,\overline \Zm_l})_{\mathfrak m}
\otimes_{\overline \Zm_l} \overline \Qm_l
- \dim_{\overline \Qm_l} H^0(X_{I,\bar s_v}, \gr^W_{i+1} \Psi_{I,v,\overline \Zm_l}\otimes V_{\xi,\overline \Zm_l})_{\mathfrak m}
\otimes_{\overline \Zm_l} \overline \Qm_l.$$

Supposons en outre que tous les $H^i(X_{I,\bar s_v}, \gr^W_r \Psi_{I,v,\overline \Zm_l}\otimes V_{\xi,\overline \Zm_l})_{\mathfrak m}$ sont 
sans torsion. On a alors une filtration de 
$H^0(X_{I,\bar s_v}, \Psi_{I,v,\overline \Fm_l}\otimes V_{\xi,\overline \Fm_l})_{\mathfrak m}$ dont les graduÈs sont les
$H^0(X_{I,\bar s_v}, \gr^W_r \Psi_{I,v,\overline \Zm_l}\otimes V_{\xi,\overline \Zm_l})_{\mathfrak m} \otimes_{\overline \Zm_l}
\overline \Fm_l$ et o˘ l'opÈrateur $\overline N_{\mathfrak m,v}$ induit, d'aprËs (\ref{eq-RZ2}), 
des isomorphismes
$$\overline N_{\mathfrak m,v}^r: H^0(X_{I,\bar s_v}, \gr^W_r \Psi_{I,v,\overline \Zm_l}\otimes V_{\xi,\overline \Zm_l})_{\mathfrak m}
\otimes_{\overline \Zm_l} \overline \Fm_l \simeq H^0(X_{I,\bar s_v}, \gr^W_{-r} 
\Psi_{I,v,\overline \Zm_l}\otimes V_{\xi,\overline \Zm_l})_{\mathfrak m} \otimes_{\overline \Zm_l} \overline \Fm_l.$$
Ainsi la dÈcomposition de Jordan de l'opÈrateur de monodromie agissant sur
$H^0(X_{I,\bar s_v}, \Psi_{I,v,\overline \Fm_l}\otimes V_{\xi,\overline \Fm_l})_{\mathfrak m}$ fournit un diagramme de Ferrers
dont le nombre de lignes de longueur $i$ est Ègal ‡
$$\begin{array}{c}
\dim_{\overline \Fm_l} H^0(X_{I,\bar s_v}, \gr^W_{i-1} \Psi_{I,v,\overline \Zm_l}\otimes V_{\xi,\overline \Zm_l})_{\mathfrak m}
\otimes_{\overline \Zm_l} \overline \Fm_l - \dim_{\overline \Fm_l} 
H^0(X_{I,\bar s_v}, \gr^W_{i+1} \Psi_{I,v,\overline \Zm_l}\otimes V_{\xi,\overline \Zm_l})_{\mathfrak m} 
\otimes_{\overline \Zm_l} \overline \Fm_l 
\\ = \\
\dim_{\overline \Qm_l} H^0(X_{I,\bar s_v}, \gr^W_{i-1} \Psi_{I,v,\overline \Zm_l}\otimes V_{\xi,\overline \Zm_l})_{\mathfrak m}
\otimes_{\overline \Zm_l} \overline \Qm_l - \dim_{\overline \Qm_l} 
H^0(X_{I,\bar s_v}, \gr^W_{i+1} \Psi_{I,v,\overline \Zm_l}\otimes V_{\xi,\overline \Zm_l})_{\mathfrak m} 
\otimes_{\overline \Zm_l} \overline \Qm_l.
\end{array}$$
Comme la reprÈsentation galoisienne $H^0(X_{I,\bar s_v}, \Psi_{I,v,\overline \Fm_l}\otimes V_{\xi,\overline \Fm_l})_{\mathfrak m}$
est isotypique relativement ‡ la reprÈsentation irrÈductible $\overline \rho_{\mathfrak m}$, le diagramme de Ferrers de 
$H^0(X_{I,\bar s_v}, \Psi_{I,v,\overline \Fm_l}\otimes V_{\xi,\overline \Fm_l})_{\mathfrak m}$  est simplement un multiple de celui de 
$d_{\mathfrak m,v}$ et donc finalement $\underline{d_{\mathfrak m,v}}=\underline{d_{\widetilde{\mathfrak m},v}}$.
\end{proof}

\subsection{Construction d'une classe de torsion}
\label{para-preuve}

Soit $\widetilde{\mathfrak m} \subset \mathfrak m$ un idÈal premier de $\Tm_{I}$ tel que,
cf. la remarque suivant la dÈfinition \ref{nota-spl2}, il existe $I \in \IC$ avec
$I_v$ un sous-groupe parahorique associÈ ‡ la partition 
$\underline{d_{\widetilde{\mathfrak m},v}}^*$. \emph{On choisit un tel $\widetilde{\mathfrak m}$
de sorte que $\underline{d_{\widetilde{\mathfrak m},v}}^*$ soit maximal.}
Rappelons que nÈcessairement
$$\underline{d_{\widetilde{\mathfrak m},v}}^* \leq \underline{d_{\mathfrak m,v}}^*,$$
et qu'en cas d'ÈgalitÈ il n'y a plus rien ‡ dÈmontrer. Supposons donc l'inÈgalitÈ prÈcÈdente stricte
de sorte que d'aprËs la proposition \ref{prop-libre-d},
des isomorphismes \ref{eq-RZ2} et de l'interprÈtation de l'opÈrateur de monodromie
$N=\log T \otimes T^\vee$ sur $\overline \Qm_l$, si tous les
$E_{1,\mathfrak m}^{p,q}$ Ètaient libres alors on aurait
$\underline{d_{\widetilde{\mathfrak m},v}}= \underline{d_{\mathfrak m,v}}$ ce qui n'est pas par 
hypothËse. Ainsi donc il existe $r$ et $i$ tel que  
$$H^i(X_{I,\bar s_v}, \gr^W_r \Psi_{I,v,\overline \Zm_l}\otimes V_{\xi,\overline \Zm_l})_{\mathfrak m,\tor} \neq (0).$$
D'aprËs \cite{boyer-invent2} et comme remarquÈ ‡ la fin du \S \ref{para-KHT}, 
les faisceaux pervers d'Harris-Taylor des $\gr^W_r(\Psi_{I,v,\overline \Qm_l})$ en niveau
parahorique, sont les $\PC_I(\chi_v,t)(\frac{\delta}{2})$. Ainsi donc il existe un
caractËre $\chi_v$ de $F_v^\times$ et un entier $t \leq d$ tels que la cohomologie
en niveau $I$ de $\PC_{I,\Gamma}(\chi_v,t)$ a de la torsion. On reprend alors les arguments du \S 2 
de \cite{boyer-mrl} dans un cas plus gÈnÈral i.e. dÈsormais $I_v$ n'est plus le sous-groupe compact maximal
mais un sous-groupe parahorique. ConsidÈrons pour ce faire $t_0$ maximal
tel qu'il existe un caractËre $\chi_v$ et $i_0 \in \Zm$ que l'on choisit minimal, pour lesquels 
$$H^{i_0}(X_{I,\bar s_v},\PC_{I,\Gamma}(\chi_v,t_0)\otimes V_{\xi,\overline \Zm_l})_{\mathfrak m,\tor} \neq (0).$$

\begin{lemm} \label{lem-tor-c}
Pour tout $t_0 < t \leq d$, la torsion des $H^i_c(X_{I,\bar s_v}^{=t},HT_{I,\Gamma}
(\chi_v,\Pi_t)\otimes V_{\xi,\overline \Zm_l})_{\mathfrak m}$
est nulle.
\end{lemm}

\rem Dans l'ÈnoncÈ ci-avant et dans la suite $HT_{I,\Gamma}(\chi_v,\Pi_t) \otimes V_{\xi,\overline \Zm_l}$
dÈsigne une structure entiËre quelconque de $HT_I(\chi_v,\Pi_t) \otimes V_{\xi,\overline \Qm_l}$, Ètant sous
entendu que le rÈsultat ne dÈpend pas du choix d'une telle structure.

\begin{proof}
CommenÁons par noter que comme $X^{=d}_{I,\bar s_v}$ est ponctuel, on a nÈcessairement
$t_0<d$. On raisonne par rÈcurrence sur $t$ du cas trivial $t=d$ ‡ $t_0+1$. Supposons donc
le rÈsultat acquis jusqu'au rang $t+1$ et traitons le cas $t$. On considËre la filtration par les poids 
de $j^{\geq t}_{!} HT_{I,\Gamma}(\chi_v,\Pi_{t})$ dont les graduÈs $gr^W_k(!,\chi_v,t)$ 
sont, d'aprËs (\ref{eq-egalite}),
nuls pour $k>0$ ou $-k<t-d$ et sinon donnÈs par $\lexp p j^{\geq (t-k)}_{!*}
HT_{I,\Gamma}(\chi_v,\Pi_{t}\{\frac{k}{2} \} \times \st_{-k}(\chi_v \{ \frac{t}{2} \})) (-k/2)$:
on rappelle, cf. (\ref{eq-p+p}), que les $p$ et $p+$ extensions intermÈdiaires coÔncident
pour les systËmes locaux d'Harris-Taylor associÈs ‡ un caractËre.
On considËre alors la suite spectrale associÈe, cf. \cite{boyer-compositio} preuve de la proposition
5.1.1:
$$E_1^{i,j}=H^{i+j}(X_{I,\bar s_v},gr_{-i}^W(!,\chi_v,t)\otimes V_{\xi,\overline \Zm_l}) 
\Rightarrow H^{i+j}(X_{I,\bar s_v},j^{\geq t}_! HT_{I,\Gamma}(\chi_v,\Pi_{t})\otimes V_{\xi,\overline \Zm_l}).$$
Le rÈsultat dÈcoule alors trivialement du fait que les $E_{1,\mathfrak m}^{i,j}$ sont
\begin{itemize}
\item sans torsion, d'aprËs la dÈfinition de $t_0$ et 

\item nuls pour $i+j \neq 0$.
\end{itemize}
\end{proof}

\begin{lemm} \label{lem-t0}
Avec les notations prÈcÈdentes, on a $i_0=0$, autrement dit pour tout $i \neq 0,1$, la torsion
de $H^i(X_{I,\bar s_v},\PC_{I,\Gamma}(\chi_v,t_0) \otimes V_{\xi,\overline \Zm_l})_{\mathfrak m}$ est triviale.
\end{lemm}

\begin{proof}
On reprend l'Ètude de la suite spectrale prÈcÈdente pour $t=t_0$:
$$E_1^{i,j}=H^{i+j}(X_{I,\bar s_v},gr_{-i}^W(!,\chi_v,t_0)\otimes V_{\xi,\overline \Zm_l}) 
\Rightarrow H^{i+j}(X_{I,\bar s_v},j^{\geq t_0}_! HT_{I,\Gamma} (\chi_v,\Pi_{t_0}) \otimes V_{\xi,\overline \Zm_l}).$$
Par dÈfinition de $t_0$, pour tout $i \neq 0$, les $E_{1,\mathfrak m}^{i,j}$ sont nuls pour $i+j \neq 0$
et sinon sans torsion. S'il existait $j<0$ tel que la torsion de $E_{1,\mathfrak m}^{0,j}$ Ètait non nul,
alors celle de $E_{\oo,\mathfrak m}^{j}=H^{j}(X_{I,\bar s_v},j^{\geq t_0}_! HT_{I,\Gamma}(\chi_v,\Pi_{t_0})
\otimes V_{\xi,\overline \Zm_l})_{\mathfrak m}$
serait aussi non nulle ce qui n'est pas puisque, $X^{=t_0}_{I,\bar s_v}$ Ètant affine, les
$H^i(X_{I,\bar s_v},j^{\geq t_0}_! HT_{I,\Gamma}(\chi_v,\Pi_{t_0})\otimes V_{\xi,\overline \Zm_l})$ sont nuls pour tout $i<0$. 
On a ainsi $i_0 \geq 0$ et on conclut en utilisant la dualitÈ de Verdier.

\end{proof}

On peut calculer les $H^i(X_{I^vv^\oo,\bar s_v},\lexp p j^{\geq t_0}_{!*} HT_{I^vv^\oo,\Gamma}(\chi_v,
\Pi_{t_0})\otimes V_{\xi,\overline \Zm_l})_{\mathfrak m}$ en utilisant la rÈsolution \ref{eq-resolution0}. On remarque alors, d'aprËs
le lemme \ref{lem-tor-c} et le fait que sur $\overline \Qm_l$ la cohomologie est concentrÈe en degrÈ $0$, que
la torsion de $H^0(X_{I^vv^\oo,\bar s_v},\lexp p j^{\geq t_0}_{!*} HT_{I^vv^\oo,\Gamma}(\chi_v,
\Pi_{t_0})\otimes V_{\xi,\overline \Zm_l})_{\mathfrak m}$ provient d'un morphisme non strict entre les $\overline \Zm_l$-modules
libres
\addtocounter{smfthm}{1}
\begin{multline} \label{eq-fleche}
H^0(X_{I^vv^\oo,\bar s_v},j^{=t_0+1}_{!*} HT_{1_{t_0}}(\chi_v, \Pi_{t_0} \{ -1/2 \} \otimes 
\chi_v \{ t_0/2 \} \otimes \Xi^{1/2})\otimes V_{\xi,\overline \Zm_l})_{\mathfrak m} \longrightarrow  \\
H^0(X_{I^vv^\oo,\bar s_v},j^{=t_0}_{!*} HT_{1_{t_0}}(\chi_v, \Pi_{t_0})\otimes V_{\xi,\overline \Zm_l})_{\mathfrak m}.
\end{multline}
Comme par hypothËse $\overline \rho_{\mathfrak m}$ est irrÈductible, en niveau infini en $v$,
les reprÈsentations automorphes $\Pi$ qui contribuent ‡ la $\overline \Qm_l$-cohomologie des deux 
termes de \ref{eq-fleche}, ont leur composante locale en $v$, d'aprËs le \S 5 de \cite{boyer-compositio},
de la forme 
$$\Pi_v \simeq \st_{t_0+1}(\chi_{v,0}) \times \st_{t_1}(\chi_{v,1}) \times \cdots \st_{t_r}(\chi_{v,r})$$
o˘ les $\chi_{v,k}$ sont des caractËres de $F_v^\times$ avec $\chi_{v,0}$ 
inertiellement Èquivalent ‡ $\chi_v$. D'aprËs loc. cit. la contribution d'une telle reprÈsentation $\Pi$ s'obtient en 
remplaÁant dans l'Ècriture prÈcÈdente de $\Pi_v$, le facteur $\st_{t_0+1}(\chi_{v,0})$ par l'induite
normalisÈe $\st_{t_0}(\chi_{v,0} \{ -1/2 \}) \times \chi_{v,0} \{ t_0/2 \}$.

\begin{lemm} 
Pour tout $t \leq t_0$, les $H^i(X_{I,\bar s_v},\lexp p j^{=t}_{!*} HT_{I,\Gamma}(\chi_v,\Pi_t)\otimes V_{\xi,\overline \Zm_l})_{\mathfrak m}$
vÈrifient les propriÈtÈs suivantes:
\begin{itemize}
\item les quotients libres sont nuls pour tout $i \neq 0$;

\item ils sont nuls pour tout $i < t_0-t$;

\item pour $i=t_0-t$, le sous-module de torsion est non nul.
\end{itemize}
\end{lemm}

\begin{proof}
Le premier point dÈcoule, comme dÈj‡ notÈ, du fait que $\overline \rho_{\mathfrak m}$ est irrÈductible.
Passons provisoirement en niveau infini en $v$ et calculons
les groupes de cohomologie de $\lexp p j^{=t}_{!*} HT_{1_t}(\chi_v,\Pi_t)$ ‡ l'aide de 
la rÈsolution \ref{eq-resolution0}. En ce qui concerne les
$H^i(X_{I^vv^\oo,\bar s_v},\lexp p j^{=t}_{!*} HT_{1_t}(\chi_v,\Pi_t)\otimes V_{\xi,\overline \Zm_l})_{\mathfrak m}$ pour $i \leq t_0-t$, 
du fait que les strates de Newton sont affines et que donc les 
$H^\delta(X_{I^vv^\oo,\bar s_v},j^{=h}_! HT(\chi_v,\Pi_h)\otimes V_{\xi,\overline \Zm_l})$ sont nuls pour $\delta<0$, seuls les
$t_0-t+2$ premiers termes de la rÈsolution interviennent, lesquels se retrouvent aussi, quitte ‡ modifier
les composantes infinitÈsimales, cf. la remarque suivant \ref{defi-HT}, dans la rÈsolution de 
$\lexp p j^{=t_0}_{!*} HT_{1_{t_0}}(\chi_v,\Pi_{t_0})$. En utilisant les propriÈtÈs d'adjonction de $j^{\geq h+1}_!$ et $i^{h+1}_*$,
les flËches 
\begin{multline*}
j^{=h+1}_! HT_{1_t}(\chi_v,\Pi_t \{ \frac{t-h-1}{2} \} \otimes \speh_{h+1-t} (\chi_v \{ \frac{t}{2} \} )) \otimes
\Xi^{\frac{h+1-t}{2}} \longrightarrow \\
j^{=h}_! HT_{1_t}(\chi_v,\Pi_t \{ \frac{t-h}{2} \} \otimes \speh_{h-t} (\chi_v \{ \frac{t}{2} \}  ))\otimes
\Xi^{\frac{h-t}{2}}
\end{multline*}
dans la rÈsolution de $\lexp p j^{=t}_{!*} HT_{1_t}(\chi_v,\Pi_t)$ se dÈduisent de celles 
\begin{multline*}
j^{=h+1}_! HT_{1_{t_0}}(\chi_v,\Pi_{t_0} \{ \frac{t_0-h-1}{2} \} \otimes \speh_{h+1-t_0} 
(\chi_v \{ \frac{t_0}{2} \} ) )\otimes \Xi^{\frac{h+1-t_0}{2}} \longrightarrow \\
j^{=h}_! HT_{1_{t_0}}(\chi_v,\Pi_{t_0} \{ \frac{t_0-h}{2} \} \otimes \speh_{h-t_0} (\chi_v \{ \frac{t_0}{2} \})) 
\otimes \Xi^{\frac{h-t_0}{2}}
\end{multline*}
‡ modification des composantes infinitÈsimales prËs.

\rem Notons, cf. \cite{dat-jl} 3.1.4, que la rÈduction modulo $l$ de $\chi_v[d]_D$ est irrÈductible, de sorte que 
$\LC(\chi_v[t]_D)$ admet un unique rÈseau stable. Il en est de mÍme pour les $\speh_t(\chi_v)$ de sorte que,
pour $\Pi_{t_0}$ bien choisi ne jouant aucun rÙle, il
n'y a pour chacun des faisceaux Ècrits dans les morphismes prÈcÈdents, qu'un unique
rÈseau stable.

D'aprËs le lemme \ref{lem-tor-c}, les groupes de cohomologie des 
$j^{=h}_! HT_{1_h}(\chi_v,\Pi_h))$ pour $h \geq t_0$ sont sans torsion de sorte que la torsion cherchÈe
ne provient que des flËches entre les 
\begin{multline*}
H^0 \Bigl (X_{I^vv^\oo,\bar s_v},  j^{=h+1}_! HT_{1_t}(\chi_v,\Pi_t \{ \frac{t-h-1}{2} \} \otimes 
\speh_{h+1-t} (\chi_v \{ \frac{t}{2} \} \otimes\Xi^{\frac{h+1-t}{2}} \otimes V_{\xi,\overline \Zm_l}\Bigr ) \\ \longrightarrow \\
H^0 \Bigl (X_{I^vv^\oo,\bar s_v}, j^{=h}_! HT_{1_t}(\chi_v,\Pi_t \{ \frac{t-h}{2} \} \otimes 
\speh_{h-t} (\chi_v \{ \frac{t}{2} \} \otimes \Xi^{\frac{h-t}{2}} \otimes V_{\xi,\overline \Zm_l} \Bigr )
\end{multline*}
et plus prÈcisÈment, du fait qu'elles sont ou non strictes. Or comme remarquÈ ci-avant, cette
propriÈtÈ se lit aussi dans la suite spectrale associÈ au calcul des groupes de cohomologie
de $\lexp p j^{=t_0}_{!*} HT_{1_{t_0}}(\chi_v,\Pi_{t_0})$, ce qui donne les propriÈtÈs
de l'ÈnoncÈ en niveau $I^vv^\oo$. Pour redescendre en niveau $I$, on utilise la suite spectrale
\addtocounter{smfthm}{1}
\begin{equation} \label{eq-niveau}
E_2^{i,j}=\ext^i(I_v,H^j(X_{I^vv^\oo,\bar s_v}, P)) \Rightarrow H^{i+j} (X_{I^vI_v,\bar s_v},P)
\end{equation}
o˘ $P$ est un faisceau pervers quelconque: les propriÈtÈs sont alors clairement vÈrifiÈes en niveau $I$.
\end{proof}

\subsection{Diminution du niveau}

On reprend les notations du paragraphe prÈcÈdent o˘ $\underline d$ est la partition associÈe au
sous-groupe parahorique $I_v$. On se propose dans un premier temps de montrer
le rÈsultat suivant.

\begin{prop} \label{prop-torsion-princ}
Sous les hypothËses du thÈorËme \ref{theo-principal}, il existe $i \in \Zm$ ainsi qu'un niveau
$J=I^vJ_v$ avec $J_v$ un sous-groupe parahorique associÈ ‡ une partition $\underline m'$
strictement plus grande que la partition $\underline m$ associÈe ‡ $I_v$, tel que
$H^i(X_{J,\bar \eta_v},V_{\xi,\overline \Zm_l})_{\mathfrak m}$ est non nul et de torsion.
\end{prop}

\begin{proof}
ConsidÈrons une reprÈsentation automorphe $\Pi$ vÈrifiant les points suivants:
\begin{itemize}
\item elle est $\xi$-cohomologique avec pour composante locale en $v$
$$\Pi_v \simeq \st_{t_0+1}(\chi_{v,0}) \times \st_{t_1}(\chi_{v,1}) \times \cdots \st_{t_r}(\chi_{v,r})$$
o˘ les $\chi_{v,k}$ sont des caractËres de $F_v^\times$ avec $\chi_{v,0}$ 
inertiellement Èquivalent ‡ $\chi_v$;

\item en niveau $I$, le morphisme (\ref{eq-fleche}) en les $\Pi^{\oo,v}$-composantes isotypiques  
n'est pas stricte. En particulier la partition associÈe ‡ $(t_0,1,t_1,\cdots,t_r)$ doit Ítre
infÈrieure ou Ègale ‡ $\underline d^*$.
\end{itemize}
On choisit une telle reprÈsentation $\Pi$ de sorte que la partition associÈe ‡
$(t_0+1,t_1,\cdots,t_r)$ soit minimale. D'aprËs la preuve du lemme prÈcÈdent, 
\begin{itemize}
\item en utilisant la suite spectrale (\ref{eq-niveau}) pour $P=\PC_{I,\Gamma}(\chi_v,1) \otimes V_{\xi,\overline \Zm_l}$, on obtient que la 
torsion de $H^{1-t_0}(X_{I^vI'_v},\PC_{I,\Gamma}(\chi_v,1) \otimes V_{\xi,\overline \Zm_l})_{\mathfrak m}$ est non nulle
o˘ $I'_v$ est un sous-groupe parahorique associÈ ‡ la partition $(d'_{\mathfrak m,v})^*$ duale de 
$(\overbrace{1,\cdots,1}^{t_0+1},t_1,\cdots,t_r)$;

\item pour tout $t >1$ et en notant $I':=I^vI'_v$, les groupes de cohomologie
$H^i(X_{I^vI'_v,\bar s_v},\PC_{I',\Gamma}(\chi_v,t)\otimes V_{\xi,\overline \Zm_l})_{\mathfrak m}$ sont sans torsion.
\end{itemize}

A partir de la torsion construite dans  $H^{1-t_0}(X_{I^vI'_v},\PC_{I',\Gamma}(\chi_v,1)\otimes V_{\xi,\overline \Zm_l})_{\mathfrak m}$,
on cherche ‡ construire de la torsion dans un des $H^i(X_{I'},V_\xi)_{\mathfrak m}$.
Pour ce faire il suffit que $I'$ contienne strictement $I$ puisqu'alors tous les
quotients libres de la suite spectrale de Rapoport-Zinkl (\ref{eq-RZ}) sont nuls.
Comme 
$$(\overbrace{1,\cdots,1}^{t_0+1},t_1,\cdots,t_r) \leq (t_0,1,t_1,\cdots,t_r) \leq \underline d^*$$ 
avec ÈgalitÈ dans la premiËre si et seulement si $t_0=1$, on en dÈduit le lemme suivant.

\begin{lemm} \label{lem-torsion2}
Si le diagramme de Ferrers ÈtiquetÈ $T_{\mathfrak m,v}$
ne contient pas deux blocs de taille $1$ d'Ètiquettes $\{ \lambda, q \lambda \}$ ou si
$t_0 >1$ alors $I'_v$ contient strictement $I_v$ i.e. $(d'_{\mathfrak m,v})^*$
est strictement plus grande que $d_{\widetilde{\mathfrak m},v}^*$.
\end{lemm}

On suppose ‡ prÈsent que $t_0=1$ de sorte que tous les 
$H^i(X_{I,\bar s_v},\PC_{I,\Gamma} (t,\chi_v)\otimes V_{\xi,\overline \Zm_l})_{\mathfrak m}$ 
sont libres dËs que $t >1$. 

\begin{lemm} \label{lem-deteriore}
Dans le cas $t_0=1$ et si $I'_v=I_v$, alors $\overline N_{\mathfrak m,v}$
n'est pas dÈtÈriorÈ relativement ‡ $\widetilde{\mathfrak m}$ au sens de la dÈfinition \ref{defi-deter}.
\end{lemm}

\begin{proof}
Notons suivant \cite{boyer-torsion}, $\Fil^1_!(\Psi_{I,v,\overline \Zm_l})$
l'image du morphisme d'adjonction
$$j^{=1}_! j^{=1,*} \Psi_{I,v,\overline \Zm_l} \longrightarrow \Psi_{I,v,\overline \Zm_l}.$$
\begin{itemize}
\item C'est un faisceau pervers qui correspond au noyau de l'opÈrateur de monodromie.

\item Son conoyau $\cofil^1_!(\Psi_{I,v,\overline \Zm_l})$ est libre, et $\PC_{I,\Gamma}(1,\chi_v)$ n'en est
pas un constituant quel que soit le caractËre $\chi_v$.
\end{itemize}
Les groupes de cohomologie $H^i(X_{I,\bar s_v},\cofil^1_!(\Psi_{I,v,\overline \Zm_l})\otimes V_{\xi,\overline \Zm_l})_{\mathfrak m}$
peuvent alors se calculer en utilisant une filtration de stratification quelconque de sorte que
ses graduÈs sont les $\PC_{I,\Gamma}(t,\chi_v)(\frac{t-1}{2}-k)$ pour $1 < t \leq d$ et $0 \leq k < t-1$.
Par hypothËse les termes $E_1^{p,q}$ de cette suite spectrale sont sans torsion et concentrÈs
sur la droite $p+q=0$. Comme dans la preuve de la proposition \ref{prop-libre-d},
par maximalitÈ de $I$ et d'aprËs (\ref{eq-RZ2b}), l'opÈrateur de monodromie sur
$H^0(X_{I,\bar s_v},\cofil^1_!(\Psi_{I,v,\overline \Zm_l})\otimes V_{\xi,\overline \Zm_l})_{\mathfrak m}$ a pour diagramme de Ferrers
un multiple de celui de $\underline{d_{\widetilde{\mathfrak m},v}}^{(1)}$ o˘ 
$\widetilde{\mathfrak m} \subset \mathfrak m$ est un idÈal premier quelconque tel qu'il existe
$\Pi \in \Pi_{\widetilde{\mathfrak m}}$ contribuant ‡ 
$H^0(X_{I,\bar s_v},\Psi_{I,v,\overline \Qm_l}\otimes V_{\xi,\overline \Zm_l})_{\mathfrak m}$. L'absence de torsion
des $H^0(X_{I,\bar s_v},\PC_{I,\Gamma}(t,\chi_v)(\frac{t-1}{2}-k)\otimes V_{\xi,\overline \Zm_l})_{\mathfrak m}$ pour tout $1 < t \leq d$ et
$0 \leq k < t-1$, nous fournit comme dans la preuve de \ref{prop-libre-d}, que le diagramme de Ferrers
associÈ ‡ $\overline N_{\mathfrak m,v}$ sur 
$$H^0(X_{I,\bar s_v},\cofil^1_!(\Psi_{I,v,\overline \Fm_l})\otimes V_{\xi,\overline \Zm_l})_{\mathfrak m} \simeq
H^0(X_{I,\bar s_v},\Psi_{I,v,\overline \Fm_l})\otimes V_{\xi,\overline \Zm_l})_{\mathfrak m}~ /~H^0(X_{I,\bar s_v},
\Fil^1_!(\Psi_{I,v,\overline \Fm_l})\otimes V_{\xi,\overline \Zm_l})_{\mathfrak m}$$
est un multiple de celui de $\underline{d_{\widetilde{\mathfrak m},v}}^{(1)}$.
Ainsi en particulier $\overline N_{\mathfrak m,v}$ n'est pas dÈtÈriorÈ.
\end{proof}
Ainsi sous les hypothËses du thÈorËme \ref{theo-principal}, on a construit une classe
de torsion dans la cohomologie en niveau $J=I^vJ_v$ avec $I_v$ strictement contenu dans $J_v$.
\end{proof}

%

\rem En reprenant la description prÈcÈdente de l'apparition de la torsion, il est aussi possible 
d'augmenter le niveau en une place de ramification de $I$ autre que $v$.

\subsection{Un ÈnoncÈ de non dÈgÈnÈrescence de la monodromie}
\label{para-dege}

On change ‡ prÈsent de point de vue: considÈrant qu'il est ‡ priori difficile d'avoir des informations sur la partition
$\underline{d_{\mathfrak m,v}}$, on cherche des conditions pour que 
$\underline{d_{\widetilde{\mathfrak m},v}}= \underline{d_{ \mathfrak m,v}}$
ou au moins pour que $\underline{d_{\mathfrak m,v}}$ ne soit pas trop \og ÈloignÈe \fg{} de $\underline{d_{\widetilde{\mathfrak m},v}}$.
L'idÈe est de reprendre les arguments prÈcÈdents, i.e. d'Ètudier la torsion dans la cohomologie des variÈtÈs de Shimura.
Pour rÈsumer la preuve prÈcÈdente, sous les hypothËses 1-2-3) du thÈorËme \ref{theo-principal}, on montre 
\begin{itemize}
\item tout d'abord que, 
quitte ‡ diminuer le niveau en $v$, un des groupes de cohomologie d'un faisceau pervers d'Harris-Taylor de la forme $\PC(1,\chi_v)$
a de la torsion 

\item puis en utilisant une des hypothËses 4), on vÈrifie que cette torsion se propage ‡ la cohomologie de la fibre gÈnÈrique
de la variÈtÈ de Shimura.
\end{itemize}
Prenons alors comme point de dÈpart des conditions explicites sur $\mathfrak m$ ou $\widetilde{\mathfrak m}$, pour que, 
quel que soit le niveau en $v$, la localisation en $\mathfrak m$ de la cohomologie de la fibre gÈnÈrique de la variÈtÈ de Shimura 
‡ coefficients dans $V_\xi$ soit sans torsion. D'aprËs le thÈorËme 4.4 de \cite{boyer-imj}, cf. aussi \cite{scholze-cara}
dans un cadre plus gÈnÈral, il suffit qu'il existe une place $w \in \spl(I)$ vÈrifiant la propriÈtÈ suivante
\addtocounter{smfthm}{1}
\begin{equation} \label{eq-torsion0}
\alpha \in S_{\mathfrak{m}}(w) \Rightarrow q_w \alpha \not \in S_{\mathfrak{m}}(w),
\end{equation}
auquel cas la localisation en $\mathfrak m$ de la cohomologie de $V_{\xi,\overline \Zm_l}$ est concentrÈe
en degrÈ mÈdian et sans torsion.

\begin{coro} \label{coro-appli}
Avec les notations  et sous les hypothËses 1-2) de \ref{theo-principal}, on suppose en outre qu'il existe une place $w \in \spl(I)$
telle l'implication de (\ref{eq-torsion0}) soit vÈrifiÈe.
Alors $\overline{N_{\mathfrak m,v}}$ n'est pas dÈtÈriorÈ relativement ‡ $\widetilde{\mathfrak m}$. 
Si en outre, cf. l'hypothËse 4-ii) de \ref{theo-principal},  la composante locale en $v$ de $\Pi_{\widetilde{\mathfrak m}}$ n'est pas de la 
forme $\chi_{v,1} \times \chi_{v,2} \times ?$ pour $\chi_{v,1}$ et $\chi_{v,2}$ des 
caractËres de $F_v^\times$ tels que $\chi_{v,2} \equiv \chi_{v,1} \nu \mod l$ alors
$$\underline{d_{\mathfrak m,v}} = \underline{d_{\widetilde{\mathfrak m},v}}.$$ 
\end{coro}

\rem Pour s'assurer que la localisation en $\mathfrak m$ de la cohomologie n'a pas de torsion, on peut aussi utiliser 
\cite{lan-suh}, et demander que 
\begin{itemize}
\item le paramËtre $\xi$ soit trËs rÈgulier au sens de la dÈfinition 7.18 de loc. cit.,
\item que $l$ soit bon, cf. la dÈfinition 2.3 de loc. cit.
\item et que le niveau $I$ soit maximal en $l$. 
\end{itemize}

\begin{proof}
On reprend les arguments du paragraphe prÈcÈdent et donc l'Ètude de la torsion dans la localisation en
$\mathfrak m$ de la suite spectrale de Rapoport-Zink.
Rappelons que $\overline \rho_{\mathfrak m}$ Ètant supposÈ irrÈductible, sur $\overline \Qm_l$, cette suite spectrale
dÈgÈnËre en $E_1$ de sorte que si aucun de ses termes $E_1^{p,q}$ n'a de la torsion alors 
$\underline{d_{\mathfrak m,v}} = \underline{d_{\widetilde{\mathfrak m},v}}$.
Si au contraire un des termes a de la torsion alors, cf. le lemme \ref{lem-t0},  il existe
$1 \leq t_0 \leq d$ que l'on choisit minimal tel que la torsion de
$H^0(X_{I,\bar s_v},\PC_{I,\Gamma}(\chi_v,t_0) \otimes V_{\xi,\overline \Zm_l})_{\mathfrak m}$ est non nulle.
En reprenant le raisonnement de la preuve de la proposition \ref{prop-torsion-princ}, on en dÈduit que, comme par hypothËse
la localisation en $\mathfrak m$ de la cohomologie de la variÈtÈ de Shimura est sans torsion, que nÈcessairement $t_0=1$.
Pour que cette torsion n'apparaisse pas dans l'aboutissement de la suite spectrale de Rapoport-Zink, il faut donc nÈcessairement,
\begin{itemize}
\item avec les notations du lemme \ref{lem-deteriore}, que $I'_v=I_v$ et donc $\overline N_{\mathfrak m,v}$
n'est pas dÈtÈriorÈ relativement ‡ $\widetilde{\mathfrak m}$,

\item et que, cf. le lemme \ref{lem-torsion2}, que la composante locale en $v$ de $\Pi_{\widetilde{\mathfrak m}}$ n'est pas de la 
forme $\chi_{v,1} \times \chi_{v,2} \times ?$ pour $\chi_{v,1}$ et $\chi_{v,2}$ des 
caractËres de $F_v^\times$ tels que $\chi_{v,2} \equiv \chi_{v,1} \nu \mod l$.
\end{itemize}
Les deux rÈsultats de l'ÈnoncÈ en dÈcoule alors trivialement.
\end{proof}

\bibliographystyle{plain}
\bibliography{bib-ok}

\def\cftil#1{\ifmmode\setbox7\hbox{$\accent"5E#1$}\else
  \setbox7\hbox{\accent"5E#1}\penalty 10000\relax\fi\raise 1\ht7
  \hbox{\lower1.15ex\hbox to 1\wd7{\hss\accent"7E\hss}}\penalty 10000
  \hskip-1\wd7\penalty 10000\box7} \def\cprime{$'$}
\begin{thebibliography}{10}

\bibitem{boyer-invent2}
P.~Boyer.
\newblock Monodromie du faisceau pervers des cycles \'evanescents de quelques
  vari\'et\'es de {S}himura simples.
\newblock {\em Invent. Math.}, 177(2):239--280, 2009.

\bibitem{boyer-compositio}
P.~Boyer.
\newblock Cohomologie des systËmes locaux de {H}arris-{T}aylor et applications.
\newblock {\em Compositio}, 146(2):367--403, 2010.

\bibitem{boyer-duke}
P.~Boyer.
\newblock La cohomologie des espaces de {L}ubin-{T}ate est libre.
\newblock {\em soumis}, 2013.

\bibitem{boyer-torsion}
P.~Boyer.
\newblock Filtrations de stratification de quelques variÈtÈs de {S}himura
  simples.
\newblock {\em Bulletin de la SMF}, 142, fascicule 4:777--814, 2014.

\bibitem{boyer-imj}
P.~Boyer.
\newblock Sur la torsion dans la cohomologie des variÈtÈs de {Sh}imura de
  {K}ottwitz-{H}arris-{T}aylor.
\newblock {\em Journal of the Institute of Mathematics of Jussieu}, pages
  1--19, 2017.

\bibitem{boyer-mrl}
P.~Boyer.
\newblock Torsion classes in the cohomology of {KHT} {S}himura's varieties.
\newblock {\em Mathematical Research Letters}, pages 1--15, 2017.

\bibitem{boyer-FT}
P.~Boyer.
\newblock Groupe mirabolique, stratification de {N}ewton raffin\'ee et
  cohomologie des espaces de {L}ubin-{T}ate.
\newblock {\em Bulletin de la SMF}, \`a paraitre.

\bibitem{scholze-cara}
A.~Caraiani and P.~Scholze.
\newblock On the generic part of the cohomology of compact unitary {S}himura
  varieties.
\newblock {\em Ann. of Math. (2)}, 186(3):649--766, 2017.

\bibitem{dat-jl}
J.-F. Dat.
\newblock Un cas simple de correspondance de {J}acquet-{L}anglands modulo $l$.
\newblock {\em Proc. London Math. Soc. 104}, pages 690--727, 2012.

\bibitem{h-t}
M.~Harris, R.~Taylor.
\newblock {\em The geometry and cohomology of some simple {S}himura varieties},
  volume 151 of {\em Annals of Mathematics Studies}.
\newblock Princeton University Press, Princeton, NJ, 2001.

\bibitem{ill}
L.~Illusie.
\newblock Autour du thÈorËme de monodromie locale.
\newblock In {\em PÈriodes $p$-adiques}, number 223 in AstÈrisque, 1994.

\bibitem{KW}
C.~Khare and J.-P. Wintenberger.
\newblock Serre's modularity conjecture.
\newblock {\em Inventiones Mathematicae}, 178 (3), 2009.

\bibitem{lan-suh}
K.W. Lan and J.~Suh.
\newblock {V}anishing theorems for torsion automorphic sheaves on compact
  {PEL}-type {S}himura varieties.
\newblock {\em Duke Math.}, 161(6):951--1170, 2012.

\bibitem{ribet}
K.~A. Ribet.
\newblock On modular representations of {${\rm Gal}(\overline{\bf Q}/{\bf Q})$}
  arising from modular forms.
\newblock {\em Invent. Math.}, 100(2):431--476, 1990.

\bibitem{y-t}
R.~Taylor and T.~Yoshida.
\newblock Compatibility of local and global {L}anglands correspondences.
\newblock {\em J.A.M.S.}, 20:467--493, 2007.

\bibitem{vigneras-induced}
M.-F. Vign{\'e}ras.
\newblock Induced {$R$}-representations of {$p$}-adic reductive groups.
\newblock {\em Selecta Math. (N.S.)}, 4(4):549--623, 1998.

\end{thebibliography}

\end{document}